\documentclass[12pt]{article}

\usepackage[T1]{fontenc}
\usepackage[utf8]{inputenc}

\usepackage{xcolor,hyperref}
\hypersetup{
    colorlinks,
    linkcolor={red!50!black},
    citecolor={green!40!black},
    urlcolor={blue!80!black}
  }
  
\usepackage{geometry}
\usepackage{amsmath,amssymb,amsthm,mathtools}
\usepackage[capitalize]{cleveref}

\newcommand{\N}{\mathbb{N}}
\newcommand{\Z}{\mathbb{Z}}
\newcommand{\R}{\mathbb{R}}
\newcommand{\C}{\mathbb{C}}
\newcommand{\Half}{\mathcal{H}}
\newcommand{\D}{\mathcal{D}}
\newcommand{\Fund}{\mathcal{F}}
\newcommand{\eps}{\varepsilon}

\newcommand{\U}{\mathcal{U}}

\newcommand{\I}{\mathcal{I}}
\newcommand{\Red}{\mathcal{R}}
\renewcommand{\S}{\mathcal{S}}
\newcommand{\A}{\mathcal{A}}
\newcommand{\V}{\mathcal{V}}

\newcommand{\from}{\colon}
\newcommand{\st}{\colon}

\DeclareMathOperator{\im}{Im}
\DeclareMathOperator{\re}{Re}
\DeclareMathOperator{\M}{\mathcal{M}}
\DeclareMathOperator{\bigO}{\mathit{O}}
\DeclareMathOperator{\FD}{FD}
\DeclareMathOperator{\Sp}{Sp}
\DeclareMathOperator{\SL}{SL}
\DeclareMathOperator{\Bstep}{BStep}
\DeclareMathOperator{\Sqrt}{sqrt}
\DeclareMathOperator{\tr}{Tr}

\DeclarePairedDelimiter{\ceil}{\lceil}{\rceil}

\DeclarePairedDelimiter{\norm}{\lVert}{\rVert}
\DeclarePairedDelimiter{\nind}{\lVert}{\rVert}
\DeclarePairedDelimiter{\abs}{|}{|}
\DeclarePairedDelimiter{\set}{\{}{\}}
\DeclarePairedDelimiter{\paren}{(}{)}

\renewcommand{\th}{\ensuremath{{}^{\mathrm{th}}}\ }
\newcommand{\first}{\ensuremath{{}^{\mathrm{st}}}\ }

\newcommand{\mat}[4]{\left(\begin{matrix}#1&#2\\#3&#4\end{matrix}\right)}

\newtheorem{prop}{Proposition}[section]
\newtheorem{lem}[prop]{Lemma}
\newtheorem{thm}[prop]{Theorem}
\Crefname{prop}{Proposition}{Propositions}
\newtheorem{cor}[prop]{Corollary}
\theoremstyle{definition}
\newtheorem{rem}[prop]{Remark}
\newtheorem{algo}[prop]{Algorithm}

\title{Certified Newton schemes for the evaluation of low-genus theta
  functions}

\author{Jean Kieffer} \date{\today}

\begin{document}

\maketitle

\begin{abstract}
  Theta functions and theta constants in low genus, especially genus~1
  and~2, can be evaluated at any given point in quasi-linear time in
  the required precision using Newton schemes based on Borchardt
  sequences. Our goal in this paper is to provide the necessary tools
  to implement these algorithms in a provably correct way. In
  particular, we obtain uniform and explicit convergence results in
  the case of theta constants in genus~1 and~2, and theta functions in
  genus~1: the associated Newton schemes will converge starting from
  approximations to~$N$ bits of precision for~$N = 60$,~$300$,
  and~$1600$ respectively, for all suitably reduced arguments. We also
  describe a uniform quasi-linear time algorithm to evaluate genus~$2$
  theta constants on the Siegel fundamental domain.  Our main tool is
  a detailed study of Borchardt means as multivariate analytic
  functions.
\end{abstract}

\section{Introduction}
\label{sec:intro}


Let~$g\geq 1$ be an integer, and let~$\Half_g$ be the Siegel upper
half space of degree~$g$, which consists of all symmetric $g\times g$
complex matrices with positive definite imaginary
part. Let~$a,b\in\{0,1\}^g$. Then the \emph{theta function} of
genus~$g$ and characteristic~$(a,b)$ is defined
on~$\C^g\times \Half_g$ by the following exponential series:
\begin{equation}
  \label{eq:theta-function}
  \theta_{a,b}(z, \tau) = \sum_{m\in \Z^g}
  \exp\left( i \pi (m + \tfrac a2)^t\tau (m+\tfrac a2)
    + 2 i \pi (m+\tfrac a2) ^t (z + \tfrac b2)\right).
\end{equation}
Theta functions appear in many areas of mathematics, from partial
differential equations to arithmetic geometry; an overview is given
in~\cite{igusa_ThetaFunctions1972, mumford_TataLecturesTheta1983,
  mumford_TataLecturesTheta1984}. They have symmetries with respect to
the action of the modular
group~$\Sp_{2g}(\Z)$~\cite[§II.5]{mumford_TataLecturesTheta1983}, and
they also satisfy the Riemann relations, a broad generalization of the
well-known duplication
formula~\cite[§II.6]{mumford_TataLecturesTheta1983}.  \emph{Theta
  constants} are the values of these functions taken at~$z=0$, and are
of interest in number theory. Each theta constant is a Siegel modular
form, and every Siegel modular form for~$\Sp_{2g}(\Z)$ has an
expression as a rational fraction in terms of theta
constants~\cite{igusa_GradedRingThetaconstants1964}; even a
polynomial, if~$g\leq 3$~\cite{igusa_GradedRingThetaconstants1964,
  freitag_VarietyAssociatedRing2019}.


In this paper, we are interested in algorithms to evaluate theta
constants at a given point~$\tau\in \Half_g$, or more generally theta
functions at a given point~$(z, \tau)$, to precision~$N$ for some
integer~$N\geq 0$. In the whole paper, we consider absolute precision:
the output will be a finitely encodable (for instance, dyadic) complex
number~$x$ such that~$\abs{\theta(z,\tau) - x} \leq 2^{-N}$.


Two main approaches to computing theta functions exist. The first one,
sometimes called the \emph{naive algorithm}, consists in computing
partial sums of the series~\eqref{eq:theta-function} and obtaining an
upper bound on the modulus of its
tail~\cite{deconinck_ComputingRiemannTheta2004,
  enge_ShortAdditionSequences2018,
  frauendiener_EfficientComputationMultidimensional2019,
  agostini_ComputingThetaFunctions2021}. The resulting algorithm can
be applied in any genus; its complexity is~$O(\M(N) N^{g/2})$
if~$(z,\tau)$ is
fixed~\cite[Prop.~4.2]{labrande_ComputingThetaFunctions2016}, and can
be made uniform in~$(z,\tau)$ if this input is suitably
reduced~\cite[Thm.~3 and Thm.~8]{deconinck_ComputingRiemannTheta2004}.


The second approach was first described by
Dupont~\cite{dupont_FastEvaluationModular2011,
  dupont_MoyenneArithmeticogeometriqueSuites2006} in the case of theta
constants of genus~$g\leq 2$. It combines the arithmetic-geometric
mean (AGM), and higher-dimensional analogues of the AGM called
Borchardt means, with Newton iterations, and claims a complexity
of~$O\paren[\big]{\M(N)\log N}$ binary operations. Extensions to theta
functions in genus~$g\leq 2$, as well as higher genera, were then
described in~\cite{labrande_ComputingJacobiTheta2018,
  labrande_ComputingThetaFunctions2016}. In practice, these algorithms
beat the naive method for precisions greater than a few hundred
thousand bits for~$g=1$, and a few thousand bits
for~\mbox{$g=2$}. This improvement is especially welcome in
number-theoretic applications, where huge precisions are often
necessary to recognize rational numbers from their complex
approximations~\cite{enge_ComputingModularPolynomials2009,
  enge_ComplexityClassPolynomial2009,
  enge_ComputingClassPolynomials2014,
  kieffer_EvaluatingModularEquations2021}, although the naive method
remains superior for~$g=1$ in the current range of practical
applications.


In order to prove the correctness of an algorithm based on Newton's
method, and establish an upper bound on its complexity, the first step
is usually to show that the linearized system that Newton's method
uses is actually invertible. A proof of this fact is currently missing
for~$g\geq
2$~\cite[§10.2]{dupont_MoyenneArithmeticogeometriqueSuites2006},
\cite[Conj.~3.6]{labrande_ComputingThetaFunctions2016}.  For~$g=1$,
the invertibility of this linear system was
proved~\cite[Prop.~11]{dupont_FastEvaluationModular2011},
\cite[Prop.~4.4]{labrande_ComputingJacobiTheta2018}, but the rate of
convergence of the resulting Newton scheme was not made explicit. This
makes these algorithms difficult to implement in a provably correct
way.


The purpose of the present paper is to turn the quasi-linear time
algorithms for theta constants in genus~$1$ and~$2$, as well as theta
functions in genus~$1$, into provably correct algorithms. This is done
by giving explicit upper bounds on derivatives of certain analytic
functions derived from Borchardt sequences on explicit polydisk
neighborhoods of the points where Newton's method is applied. 
In the case of genus~$2$ theta constants, we also show how to combine
Newton's method with the naive algorithm to obtain a uniform
quasi-linear complexity on the Siegel fundamental domain, thus
generalizing earlier constructions in
genus~$1$~\cite[Thm.~5]{dupont_FastEvaluationModular2011},
\cite[§4.2]{labrande_ComputingJacobiTheta2018}. In the case of theta
functions in genus~$2$, and higher genera, we are no longer able to
prove that Newton's method will succeed for all inputs. However, if it
succeeds, then the same methods can be applied to certify the
correctness of the result.


The paper is organized as follows. In \cref{sec:newton}, we give a
general result of explicit convergence for Newton schemes involving
multivariate analytic functions. We study Borchardt means as analytic
functions in detail in \cref{sec:borchardt}. In \cref{sec:dupont}, we
review the existing Newton schemes for the computation of theta
functions; then, we obtain explicit values for the magnitudes and
radii of convergence of the analytic functions defining them, and thus
explicit convergence results. Finally, we present the uniform
algorithm to compute genus~$2$ theta constants in \cref{sec:unif}.

\section{Certified multivariate Newton iterations}
\label{sec:newton}

In this section, we are interested in designing provably correct
Newton schemes for multivariate analytic functions, assuming that the
system is linearized using finite differences at each step.  More
precisely, let~$\U$ be an open set in~$\C^r$, let~$f\from \U\to \C^r$
be an analytic function, and let~$x_0\in \U$; assuming that~$f(x_0)$
is known and that~$f$ can be evaluated at any point, we are interested
in building a Newton scheme to compute~$x_0$ itself.

First, we give an explicit convergence result provided that the first
and second derivatives of~$f$ are locally bounded around~$x_0$, and
that~$df(x_0)$ is invertible. Using Cauchy's formula, we also obtain
explicit convergence estimates if we simply assume that~$f$ is bounded
on a certain polydisk around~$x_0$. Finally, we translate these
theoretical results into the concrete world of finite-precision
arithmetic. All these results are certainly well-known in spirit, but
we were unfortunately unable to find sufficiently explicit results in
the literature.

Let us introduce some notation. We always consider~$\C^r$ as a normed
vector space for the~$L^\infty$ norm, denoted simply
by~$\norm{\cdot}$: in terms of coordinates, we have
\begin{displaymath}
  \norm{(x_1,\ldots,x_r)} = \max_{1\leq j\leq r} \abs{x_j}.
\end{displaymath}
If~$\rho>0$ and~$x\in \C^r$, we denote by~$\D_\rho(z)$ the open ball
(i.e.~the polydisk) centered in~$z$ of radius~$\rho$.  We also denote the induced
norm of (multi-)linear operators by~$\nind{\cdot}$.

Let~$(e_i)$ be the canonical basis of~$\C^r$, and denote the
coordinates by~$x_1,\ldots, x_r$. If~$x\in \U$, then we have
\begin{displaymath}
  df(x) = \sum_{i=1}^r \frac{\partial f}{\partial x_i}(x) \,dx_i,
\end{displaymath}
where~$dx_i$ is seen as the linear form~$x\mapsto x_i$.
For~$\eta>0$ such that~$\D_\eta(x) \subset \U$, we also define
\begin{displaymath}
  \FD_\eta f(x) = \sum_{i=1}^r \frac{f(x+\eta e_i) - f(x)}{\eta} \,dx_i.
\end{displaymath}
This linear operator is an approximation of~$df(x)$ using finite
differences.

Assume we already know~$x\in \U$ such that~$\norm{x-x_0}\leq \eps$ for
some~$\eps>0$. Then we can formulate a Newton iteration step to refine
the approximation~$x$ of~$x_0$ as follows: simply replace~$x$
by~$x + h$, where
\begin{displaymath}
  h = df(x)^{-1} (f(x_0)-f(x)).
\end{displaymath}
In the finite differences version, we take instead:
\begin{displaymath}
  h = \FD_\eta f (x)^{-1} (f(x_0)-f(x)),
\end{displaymath}
where~$\eta>0$ is a suitably chosen small parameter. Then, provided
that~$\eps$ is small enough,~$\norm{x+h-x_0}$ will be of the order
of~$\eps^2$, ensuring quadratic convergence of the Newton
iteration.

\begin{prop}
  \label{prop:theoretical-newton}
  Let~$\U\subset\C^r$ be an open set, let~$f\from \U\to\C^r$ be an
  analytic function, and let~$x_0\in\U$. Let~$\rho>0$
  and~$B_1,B_2,B_3 \geq 1$ be real numbers such
  that~$\D_\rho(x_0)\subset\U$ and the following inequalities are
  satisfied:  
  \begin{enumerate}
  \item $\nind{df(x)}\leq B_1$ and~$\nind{d^2f(x)}\leq B_2$ for all~$x\in \D_\rho(x_0)$;
  \item $df(x_0)$ is invertible and $\nind{df(x_0)^{-1}}\leq B_3$.
  \end{enumerate}
  Let~$\eps,\eta>0$ be such that
  \begin{displaymath}
    \eps \leq \min\set[\Big]{\frac{\rho}{2}, \frac{1}{2B_2B_3}}
    \quad \text{and}\quad
    \eta \leq \frac{\eps}{4r B_1 B_3}. 
  \end{displaymath}
  Then, for each~$x\in \C^r$ such that~$\norm{x - x_0}\leq \eps$, if
  we set
  \begin{displaymath}
    h = \FD_\eta f (x)^{-1} (f(x_0)-f(x)),
  \end{displaymath}
  we will have
  \begin{displaymath}
    \norm{x + h - x_0} \leq 2 B_2 B_3 \eps^2.
  \end{displaymath}
\end{prop}

\begin{proof}
  First, note that
  \begin{displaymath}
    \nind{df(x) - df(x_0)} \leq B_2 \norm{x-x_0} \leq B_2\eps \leq \frac{1}{2\nind{df(x_0)^{-1}}},
  \end{displaymath}
  so~$df(x)$ is also invertible, with~$\nind{df(x)^{-1}}\leq 2B_3$.
  We can now study the ``usual'' Newton scheme. Let us write
  \begin{displaymath}
    f(x_0) = f(x) + df(x)(x_0-x) + v,
  \end{displaymath}
  for some vector~$v$ such that~$\norm{v}\leq \frac12 B_2\eps^2$.
  Let~$h_0 = df(x)^{-1}(f(x_0)-f(x))$. Then
  \begin{equation}
    \label{eq:x+h0-x0}
    \norm{x+h_0 - x_0} = \norm{df(x)^{-1} v} \leq B_2 B_3 \eps^2.
  \end{equation}
  Finally, we show that~$h$ is close
  to~$h_0$. Since~$\D_\eta(x)\subset \D_\rho(x_0)$
  (because~$\eta\leq \eps\leq \rho/2$), we have for
  each~$1\leq j\leq r$:
  \begin{displaymath}
    \abs[\Big]{\frac{f(x+\eta e_i) - f(x)}{\eta}
      - \frac{\partial f}{\partial x_j}(x)} \leq \frac12 B_2\eta.
  \end{displaymath}
  Therefore,
  \begin{displaymath}
    \nind{\FD_\eta f(x) - df(x)} \leq \frac{r}{2} B_2 \eta \leq \frac{1}{4 B_3}
    \leq \frac{1}{2\nind{df(x)^{-1}}},
  \end{displaymath}
  so that
  \begin{displaymath}
    \nind{\FD_\eta f(x)^{-1} - df(x)^{-1}} \leq 2 \nind{df(x)^{-1}}^2 \cdot \nind{\FD_\eta f(x) - df(x)}
    \leq 4r B_2 B_3^2 \eta,
  \end{displaymath}
  and
  \begin{equation}
    \label{eq:h-h0}
    \norm{h - h_0} \leq 4r B_2 B_3^2 \eta \norm{f(x) - f(x_0)}
    \leq 4r B_1B_2B_3^2 \eta \eps \leq B_2 B_3 \eps^2.
  \end{equation}
  We obtain the result from~\eqref{eq:x+h0-x0},~\eqref{eq:h-h0}, and
  the triangle inequality.
\end{proof}

Cauchy's integration
formula~\cite[Thm.~2.2.1]{hormander_IntroductionComplexAnalysis1966}
provides uniform upper bounds on~$\nind{df(x)}$ and~$\nind{d^2f(x)}$
for~$x\in \D_\rho(x_0)$ whenever a uniform upper bound on~$\norm{f}$
on a slightly larger polydisk is known; this makes the necessary data
in \cref{prop:theoretical-newton} easier to collect.

\begin{prop}
  \label{prop:cauchy}
  Let~$r,s\geq 1$, let~$x_0\in \C^r$, let~$\rho>0$, and
  let~$f\from \D_\rho(x_0)\to \C^s$ be an analytic
  function. Let~$M\geq 0$ such that~$\norm{f(x)}\leq M$ for
  all~$x\in D_\rho(x_0)$. Then for every~$n\geq 0$ and
  every~$x\in \D_{\rho/2}(x_0)$, we have
  \begin{displaymath}
    \nind{d^nf(x)} \leq \frac{2^n n!}{\rho^n} \binom{n+r}{r} M.
  \end{displaymath}
\end{prop}

\begin{proof}
  It is enough to prove that
  \begin{displaymath}
    \nind{d^n f(x_0)} \leq \frac{n!}{\rho^n} \binom{n+r}{r} M
  \end{displaymath}
  for all~$n$; afterwards, we simply note
  that~$D_{\rho/2}(x)\subset \D_\rho(x_0)$ for
  each~$x\in \D_{\rho/2}(x_0)$. Write~$x_0 = (z_1,\ldots, z_r)$. We
  compute the Taylor expansion of~$f$ at~$x_0$ using Cauchy's formula.
  For each~$\zeta = (\zeta_1,\ldots,\zeta_r)\in \D_{\rho/2}(x_0)$, we
  have
  \begin{displaymath}
    f(\zeta) = \sum_{n = (n_1,\ldots, n_r)\in \N^r} a_n(f) \prod_{j=1}^r (\zeta_j - z_j)^{n_j},
  \end{displaymath}
  where the Taylor coefficients~$a_n(f)\in \C^s$ are computed as
  follows:
  \begin{displaymath}
    a_n(f) = \frac{1}{(2\pi i)^r} \int_{\partial \D_\rho(z_1)}\cdots \int_{\partial \D_\rho(z_r)}
    \frac{f(x_1,\ldots,x_r)}{\prod_{j=1}^r (x_j - z_j)^{n_j+1}}
    \ dx_1\cdots dx_r.
  \end{displaymath}
  In particular,
  \begin{displaymath}
    \norm{a_n(f)} \leq \frac{M}{\rho^{\sum_j n_j}}.
  \end{displaymath}
  Now, for each~$v\in \C^r$, the value of~$d^nf(x_0)(v,\ldots,v)\in \C^s$ is
  given by all terms of total degree~$n$ in the Taylor expansion, up
  to a factor of~$n!$:
  \begin{displaymath}
    d^nf(x_0)(v,\ldots,v) = n! \sum_{m\in \N^r,\ \sum m_j = n} a_m(f) \prod_{j=1}^r v_j^{m_j}.
  \end{displaymath}
  There are exactly~$\binom{n+r}{r}$ terms in the sum.
  Since~$d^n f(x_0)$ is a symmetric operator, the result follows
  easily.
\end{proof}

In order to run certified Newton iterations on a computer, showing a
theoretical convergence result is not enough: we also have to consider
precision losses, which will for instance prevent us from
choosing~$\eta$ too close to zero. Thankfully, Newton iterations are
self-correcting, and precision losses can be controlled by taking an
additional, explicit safety margin.

We adopt the following computational model for complex numbers. Dyadic
elements of~$\C^r$ (i.e.~elements of~$2^{-N}\Z[i]^r$ for
some~$N\in \Z$) are represented exactly; and for a
general~$z\in \C^r$, we call an \emph{approximation of~$z$ to
  precision~$N$} a dyadic~$z'$ such that~$\norm{z-z'}\leq
2^{-N}$. Elementary operations on approximations of complex numbers
can be carried out using ball
arithmetic~\cite{tucker_ValidatedNumericsShort2011}. Recall that a
function~$C\from \Z_{\geq 1}\to \R_{\geq 0}$ is called
\emph{superlinear} if~$C(m+n)\geq C(m)+C(n)$ for
all~$m,n\in \Z_{\geq 1}$.

\begin{thm}
  \label{thm:practical-newton}
  Let~$\U\subset\C^r$ be an open set, let~$f\from \U\to\C^r$ be an
  analytic function, and let~$x_0\in\U$. Let~$\rho\leq 1,M \geq 1,$
  and~$B_3\geq 1$ be real numbers such that~$D_\rho(x_0)\subset \U$,
  $\nind{f(x)}\leq M$ for each~$x\in \D_\rho(x_0)$,
  and~$\nind{df(x_0)^{-1}}\leq B_3$. Let~$C\from \Z_{\geq 1}\to \R$ be
  a superlinear function such that the following holds:
  \begin{itemize}
  \item there exists an algorithm~$\A$ which, given an
    exact~$x\in \D_\rho(x_0)$ and~$N\geq 0$, computes an approximation
    of~$f(x)$ to precision~$N$ in~$C(N)$ binary operations;
  \item two~$N$-bit integers can be multiplied in~$C(N)$ binary
    operations;
  \item we have~$C(2N)\leq KC(N)$ for some~$K\geq 1$ and for all~$N$
    sufficiently large.
  \end{itemize}
    Then, given~$N\geq 0$,
  an approximation of~$f(x_0)$ to precision~$N$, and an approximation
  of~$x_0$ to precision
  \begin{displaymath}
    n_0 = 2\ceil[\big]{\log_2(2(r+1)M/\rho)} + 2\ceil[\big]{\log_2(B_3)} + 4,
  \end{displaymath}
  \cref{algo:newton} below computes an approximation of~$x_0$ to
  precision~$N - \ceil[\big]{\log_2(B_3)} - 1$
  in~$\bigO\paren[\big]{C(N)}$ binary operations; the hidden constant
  in this complexity bound depends only on~$r,\rho,M,B_3,$ and~$K$.
\end{thm}

We now describe the algorithm. Let
\begin{displaymath}
  B_1 = \frac{2(r+1)M}{\rho}
  \quad\text{and}\quad
  B_2 = \frac{2(r+1)(r+2)M}{\rho^2}.
\end{displaymath}
By \cref{prop:cauchy}, the real numbers $\rho/2,B_1,B_2,B_3$ meet the
conditions of \cref{prop:theoretical-newton}. Up to decreasing~$\rho$
and increasing~$B_1,B_2,B_3$, we may assume that they are all powers
of~$2$. Denote the given dyadic approximation of~$f(x_0)$ by~$z_0$.

\begin{algo}[Certified Newton iterations for analytic functions]~
  \label{algo:newton}
  \begin{enumerate}
  \item Let~$n = n_0$, and let~$x$ be the given dyadic approximation
    of~$x_0$ to precision~$n$.
  \item While~$n < N$, do:
    \label{step:loop}
    \begin{enumerate}
    \item
      Let~$m = n + \log_2(B_1) + \log_2(B_3) + \ceil{\log_2(r)} + 2$,
      and~$\eta = 2^{-m}$;
    \item Using algorithm~$\A$, compute approximations of~$f(x)$
      and~$f(x+\eta e_j)$ for all~$1\leq j\leq r$ to precision
      $p = 2n + 2\ceil{\log_2(r)} + 2\log_2(B_1) + 2\log_2(B_3) + 9$;
    \item Compute an approximation of the~$r\times r$ matrix~$M_1$
      whose~$j$th column contains the finite
      difference~$\frac 1\eta \paren[\big]{f(x+\eta e_j)-f(x)}$, for
      all~$j$, to precision $p - \log_2(1/\eta)-1$ (entrywise);
    \item
      \label{step:inverse}
      Compute an approximation of the~$r\times r$ matrix~$M_2 = M_1^{-1}$ to
      precision $p' = p - \log_2(1/\eta) - 2\log_2(B_3) - 7$;
    \item
      \label{step:h}
      Compute an approximation of the
      vector~$h = M_2\paren[\big]{z_0 - f(x)}$ to precision
      $p' + n - 1 - \log_2(B_1) - \ceil{\log_2(r)}$;
    \item Let~$n' = 2n - \log(B_2) - \log(B_3) - 2$; replace~$x$ by a
      dyadic approximation of~$x+h$ to precision~$n'+1$, and
      replace~$n$ by~$n'$.
    \end{enumerate}
  \item Return~$x$.
  \end{enumerate}
\end{algo}


\begin{proof}[Proof of \cref{thm:practical-newton}] We will show that
  the different quantities appearing in \cref{algo:newton} can be
  computed to the claimed precisions, and that~$x$ remains an
  approximation of~$f^{-1}(z_0)$ to
  precision~$n$. Since~$\nind{df(x)^{-1}}\leq 2B_3$ for
  all~$x\in \D_\rho(x_0)$, the result will be an approximation
  of~$x_0$ to precision~$N - \log_2(B_3) - 1$, as claimed.

  At the beginning of each loop,~$x$ is dyadic, and so are
  the~$x+\eta e_j$ for each $1\leq j\leq r$. Therefore, each entry
  of~$M_1$ can be computed to precision~$p - \log_2(1/\eta) - 1$. Note
  that~$\nind{\FD_\eta f(x)^{-1}}\leq 4B_3$ as a linear
  operator. Let~$M_1'$ be a dyadic approximation of~$M_1$ to
  precision~$p - \log_2(1/\eta)$; then we have
  \begin{displaymath}
    \nind{M_1 - M_1'} \leq \frac{1}{2\nind{M_1^{-1}}},
  \end{displaymath}
  so
  that~$\nind{M_1^{-1}- M_1'^{-1}} \leq 2 \nind{M_1^{-1}}^2 \nind{M_1
    - M_1'} \leq 32 B_3^2\, 2^{-p}/\eta$. This shows that~$M_1^{-1}$
  can be computed to the required precision~$p'$ in
  step~\eqref{step:inverse}. In step~\eqref{step:h}, we perform the
  matrix-vector product using the schoolbook formula. The entries
  of~$M_2$ have modulus at most~$4 B_3$, and are known up to
  precision~$p'$; the entries of~$z_0-f(x)$ have modulus at
  most~$2^{-n}B_1$, and are known up to precision~$p-1$. The total
  error on the product can be bounded above by
  \begin{displaymath}
    r(4B_3 2^{-p+1} + 2^{-n} B_1 2^{-p'} + 2^{-p'-p-1}) \leq 2^{-n+1}r B_12^{-p'}.
  \end{displaymath}
  The precision~$p$ was chosen in such a way that we obtain, at the
  end of the loop, an approximation of~$x+h$ to
  precision~$2n - \log(B_3) - 1 \geq n'+1$. By
  \cref{prop:theoretical-newton}, the result is also an approximation
  of~$f^{-1}(z_0)$ to precision~$n'$.

  The initial value of~$n_0$ ensures that~$n' > 3n/2$, so that number
  of steps in the loop is~$O(\log N)$. Each loop involves a finite
  number of elementary operations with complex numbers of
  modulus~$O(1)$ at precision~$2n + O(1)$, where the hidden constants
  depend only on~$r,\rho, M$, and~$B_3$; the cost of these
  computations is~$O\paren[\big]{C(n)}$ binary operations. Since~$C$
  is superlinear, the cost of the last loop dominates the cost of the
  whole algorithm, a well-known feature of Newton's method.
\end{proof}

\section{Borchardt means as analytic functions}
\label{sec:borchardt}

The existing Newton schemes for the computation of theta
functions~\cite{dupont_FastEvaluationModular2011,
  dupont_MoyenneArithmeticogeometriqueSuites2006,
  labrande_ComputingJacobiTheta2018,
  labrande_ComputingThetaFunctions2016} are based on \emph{Borchardt
  means}, a higher-dimensional analogue of the classical
arithmetic-geometric mean
(AGM)~\cite{cox_ArithmeticgeometricMeanGauss1984}. Additional
references for the study of Borchardt means, especially in genus~$2$,
are~\cite{bost_MoyenneArithmeticogeometriquePeriodes1988,
  jarvis_HigherGenusArithmeticgeometric2008}.  Our goal in this
section is to study Borchardt means as analytic functions in detail,
obtaining explicit bounds on their magnitudes and radii of
convergence.

\subsection{Borchardt sequences}
\label{subsec:borchardt-def}

Fix~$g\geq 1$, and let~$\I_g = (\Z/2\Z)^g$. A \emph{Borchardt
  sequence} of genus~$g$ is by definition a sequence of complex
numbers
\begin{displaymath}
  s = \paren[\big]{s_b^{(n)}}_{b\in \I_g, n\geq 0}
\end{displaymath}
that satisfy the following recurrence relation: for every~$n\geq 0$,
there exists a choice of square roots~$\paren[\big]{t_b^{(n)}}_{b\in \I_g}$
of~$\paren[\big]{s_b^{(n)}}_{b\in \I_g}$ such that for all~$b\in \I_g$, we have
\begin{equation}
  \label{eq:borchardt-step}
  s_b^{(n+1)} = \frac{1}{2^g} \sum_{b_1 + b_2 = b} t_{b_1}^{(n)} t_{b_2}^{(n)}.
\end{equation}
We say that~$\paren[\big]{s_b^{(n+1)}}_{\smash{b\in \I_g}}$ is the
result of a Borchardt step given by the choice of square
roots~$\paren[\big]{t_b^{(n)}}_{b\in\I_g}$ at the~$n$\th term.  This
recurrence relation emulates the duplication formula satisfied by
theta constants~\cite[p.\,214]{mumford_TataLecturesTheta1983}, after
identifying~$\smash{\set{0,1}^g}$ with~$\I_g$ in the natural way: for
every~$\tau\in \Half_g$, the sequence of squared theta constants
\begin{equation}
  \label{eq:theta-as-borchardt}
  \paren[\big]{\theta^2_{0,b}(0,2^n\tau)}_{b\in \I_g, n\geq 0}
\end{equation}
is a Borchardt sequence.

The convergence behavior of Borchardt sequences is similar to that of
the classical
AGM~\cite[§7.2]{dupont_MoyenneArithmeticogeometriqueSuites2006}. Let
us define a set of complex numbers to be \emph{in good position} if it
is included in an open quarter plane seen from the origin, i.e.~a set
of the form
\begin{displaymath}
  \set[\big]{r \exp(i\theta) \st r>0, \alpha < \theta < \alpha + \tfrac{\pi}{2}}
\end{displaymath}
for some angle~$\alpha\in \R$. We say that the~$n$\th step of a
Borchardt sequence is given by \emph{good sign choices} (or for short,
is \emph{good}) if the square
roots~$\paren[\big]{t_b^{(n)}}_{\smash{b\in\I_g}}$ are in good position;
otherwise, we say that this step is \emph{bad}. Then a Borchardt
sequence~$s$ will converge to~$(0,\ldots,0)$ if and only if~$s$
contains infinitely many bad steps. On the other hand, a Borchardt
sequence~$s$ in which all steps are good after a while converges to a
limit of the form~$(\mu,\ldots,\mu)$ for some~$\mu\neq 0$, and the
speed of convergence is quadratic; we call~$\mu = \mu(s)$ the
\emph{Borchardt mean} of the sequence.  Borchardt sequences given by
theta functions as in~\eqref{eq:theta-as-borchardt} are of this second
type: see for
instance~\cite[Prop.~6.1]{dupont_MoyenneArithmeticogeometriqueSuites2006}.

A related kind of recurrent sequence is used in the context of
computing theta functions. Let us call an \emph{extended Borchardt
  sequence} of genus~$g$ a pair~$(u,s)$ of sequence of complex numbers
\begin{displaymath}
  u = \paren[\big]{u_b^{(n)}}_{b\in \I_g, n\geq 0},\qquad
  s = \paren[\big]{s_b^{(n)}}_{b\in \I_g, n\geq 0}
\end{displaymath}
satisfying the following recurrence relation: for every~$n\geq 0$,
there exists a choice of square
roots~$\paren[\big]{v_b^{(n)}}_{b\in \I_g}$
of~$\paren[\big]{u_b^{(n)}}_{b\in \I_g}$
and~$\paren[\big]{t_b^{(n)}}_{b\in \I_g}$
of~$\paren[\big]{s_b^{(n)}}_{b\in \I_g}$ such that for all~$b$,
\begin{equation}
  \label{eq:ext-borchardt-step}
  u_b^{(n+1)} = \frac{1}{2^g} \sum_{b_1 + b_2 = b} v_{b_1}^{(n)} t_{b_2}^{(n)}
  \quad \text{and}\quad
  s_b^{(n+1)} = \frac{1}{2^g} \sum_{b_1 + b_2 = b} t_{b_1}^{(n)} t_{b_2}^{(n)}.
\end{equation}
In particular,~$s$ is a regular Borchardt sequence. We say that
the~$n$\th step in~$(u,s)$ is \emph{good} if both of the
sets~$\paren[\big]{v_b^{(n)}}_{b\in \I_g}$
and~$\paren[\big]{t_b^{(n)}}_{b\in \I_g}$ are independently in good
position, and \emph{bad} otherwise. For each~$\tau\in \Half_g$
and~$z\in \C^g$, the duplication formula for theta functions implies
that the sequence
\begin{displaymath}
  \paren[\big]{\theta^2_{0,b}(z, 2^n\tau), \theta^2_{0,b}(0, 2^n\tau)}_{b\in \I_g, n\geq 0} 
\end{displaymath}
is an extended Borchardt sequence; it contains only finitely many bad
steps as well.

It is not true in general that an extended Borchardt sequence
containing finitely many bad steps converges quadratically. Instead,
following~\cite{labrande_ComputingThetaFunctions2016}, we define the
\emph{extended Borchardt mean} of such a sequence~$(u,s)$ to be
\begin{equation}
  \label{eq:def-ext-borchardt}
  \lambda(u, s) = \mu(s) \cdot
  \lim_{n\to +\infty} \left(\frac{u_0^{(n)}}{s_0^{(n)}}\right)^{2^n}
  = \mu(s) \cdot
  \lim_{n\to +\infty} \left(\frac{u_0^{(n)}}{\mu(s)}\right)^{2^n}.
\end{equation}
These associated sequences do converge
quadratically~\cite[Prop.~3.7]{labrande_ComputingThetaFunctions2016}.


Assume that we are given a Borchardt sequence~$s$ containing finitely
many bad steps. Then we may try to construct a function~$\mu_s$,
defined at any point~$\smash{x = (x_b)_{b\in \I_g}}$ in some
neighborhood of~$\paren[\big]{s_b^{(0)}}_{\smash{b\in \I_g}}$, by the following procedure:
``construct a modified Borchardt sequence whose first term is~$(x_b)$
that follows same choices of square roots as in~$s$, and take its
Borchardt mean''. The Newton schemes we want to study are precisely
built around this kind of functions~$\mu_s$, and their analogues for
extended Borchardt means. In the rest of this section, we show that
these functions indeed exist as analytic functions defined on explicit
polydisks, provided that all terms in the relevant Borchardt sequences
are bounded away from zero.

\subsection{The case of good sign choices}
\label{subsec:borchardt-good}

Let~$s$ be a Borchardt sequence containing good steps only. Then we
can find real numbers~$0< m_0 < M_0$ and~$\alpha$ such that such that
the first term~$\paren[\big]{s_b^{(0)}}_{\smash{b\in \I_g}}$ of~$s$ lies in the open
set~$\U_g(m_0, M_0)$ of~$\C^{2^g}$ defined as follows:
\begin{displaymath}
  \U_g(m_0, M_0) = \bigcup_{\smash{\alpha\in [0,2\pi]}} \U_g(m_0,M_0,\alpha),
\end{displaymath}
where
\begin{displaymath}
  \U_g(m_0, M_0, \alpha) = \set[\big]{(x_b)_{b\in \I_g} \st
    \forall b\in \I_g, m_0 < \re(e^{-i\alpha}x_b) < M_0}.
\end{displaymath}

\begin{prop}
  \label{prop:good-borchardt-analytic}
  Let~$0< m_0 < M_0$ be real numbers. Then there exists a unique
  analytic function~$\mu \from \U_g(m_0, M_0)\to \C$ with the
  following property: for every point
  $x=(x_b)_{b\in \I_g} \in \U_g(m_0, M_0)$, the value of~$\mu$ at~$x$
  is the Borchardt mean of the unique Borchardt sequence with first
  term~$x$ given by good steps only.  Moreover, the inequalities
  $m_0 \leq \abs{\mu(x)} \leq M_0$ hold for all~$x\in \U_g(m_0,M_0)$.
\end{prop}

\begin{proof}
  For each~$x\in \U_g(m_0, M_0,\alpha)$, there is a unique way of
  making a good Borchardt step starting from~$x$; moreover the result
  of this Borchardt step still lands in~$\U_g(m_0, M_0,\alpha)$
  by~\cite[Lem.~7.3]{dupont_MoyenneArithmeticogeometriqueSuites2006}. Therefore
  we may define~$\mu(x)$ as the limit of the resulting Borchardt
  sequence; we have~$m_0\leq \abs{\mu(x)}\leq M_0$. Since there exists
  an analytic square root function on~$\U_1(m_0,M_0,\alpha)$, the
  function~$\mu$ on~$\U_g(m_0,M_0,\alpha)$ is the pointwise limit of a
  sequence of analytic functions. The convergence is uniform on
  compact sets
  by~\cite[Prop.~7.2]{dupont_MoyenneArithmeticogeometriqueSuites2006},
  so~$\mu$ is analytic on the whole of~$\U_g(m_0,M_0)$.
\end{proof}

We now consider the case of extended Borchardt means given by good
choices of square roots only. This case is easier to analyse if we
assume that the truly Borchardt part of the sequence already starts in
the quadratic convergence
area. By~\cite[Prop.~7.1]{dupont_MoyenneArithmeticogeometriqueSuites2006},
if we have
\begin{equation}
  \label{eq:quadratic-area}
  \abs[\big]{s_b^{(n)} - s_0^{(n)}} < \tfrac{\eps}{4} \abs[\big]{s_0^{(n)}}
\end{equation}
for some~$\eps \leq 1/2$, then we have
\begin{displaymath}
  \abs[\big]{s_b^{(n+k)} - s_0^{(n+k)}} \leq
  \frac{2}{7}\left(\frac{7\eps}{8}\right)^{2^k}\cdot \max_{b\in \I_g}\ \abs[\big]{s_b^{(n)}}
\end{displaymath}
for all~$k\geq 0$ and~$b\in \I_g$. If we assume that the first term
of~$s$ lies in a ball of the form $\D_{\rho}(z_0)$ for
some~$z_0\in \C^\times$ and $0< \rho < \tfrac{1}{17} \abs{z_0}$, then
inequality~\eqref{eq:quadratic-area} will be satisfied
with~$\eps = \frac12$ at~$n=0$.

\begin{prop}
  \label{prop:good-ext-borchardt-analytic}
  Let~$0< m_0<M_0$ be real numbers, fix a nonzero~$z_0\in \C$,
  and let $0 < \rho < \tfrac{1}{17} \abs{z_0}$.  Then there exists a
  unique analytic function
  \begin{displaymath}
    \lambda\from \U_g(m_0,M_0)\times \D_\rho(z_0)^{2^g}
    \to \C
  \end{displaymath}
  with the following property: for every~$(x,y)$ in this open
  set,~$\lambda(x,y)$ is equal to the extended Borchardt mean of any
  extended Borchardt sequence with first term~$(x,y)$ given by good
  steps only. Moreover, we have
  \begin{displaymath}
    \exp\left(-28\log^2(4M/m)\right) \leq \abs{\lambda(x,y)}
    \leq \exp\left(20\log^2(4M/m)\right)
  \end{displaymath}
  where~$M = \max\{\abs{z_0}+\rho, M_0, 1\}$
  and~$m = \min\{\abs{z_0}-\rho, m_0, 1\}$.
\end{prop}

\begin{proof}
  We follow the proof
  of~\cite[Thm.~3.10]{labrande_ComputingJacobiTheta2018}, and hints on
  how to generalize it to higher genera given
  in~\cite[Prop.~3.7]{labrande_ComputingThetaFunctions2016}. We may
  fix~$\alpha\in\R$ and restrict our attention
  to~$\U_g(m_0,M_0,\alpha)$.

  First of all, by the proof
  of~\cite[Lem.~3.8]{labrande_ComputingJacobiTheta2018},
  each~$(x,y)\in \U_g(m_0,M_0,\alpha) \times \D_\rho(z_0)^{2^g}$ is
  the starting point of at least one extended Borchardt
  sequence~$(u,s)$ with good sign choices at all steps. Any two such
  sequences differ at the~$n$\th term by global
  multiplication~$\paren[\big]{u_b^{(n)}}_{\smash{b\in\I_g}}$ by a~$2^n$-th
  root of unity; therefore, their extended Borchardt means are
  equal. Note that~$M$ (resp.~$m$) is an upper (resp.~lower) bound on
  the modulus of all complex numbers appearing in these extended
  Borchardt sequences.  In the rest of this proof, we
  fix~$\theta_0\in\R$ such that~$\theta_0 = \arg(z_0)$ mod~$2\pi$, and
  consider the unique such sequence~$(u,s)$ whose $n$\th term lies in
  \begin{displaymath}
    \U_g\left(m,M,\frac{\alpha + (2^n-1)\theta_0}{2^n}\right) \times
    \U_g(m,M,0).
  \end{displaymath}
  Each term of~$(u,s)$ is an analytic function of its starting
  point~$(x,y)$.

  By construction, we have for all~$n\geq 0$:
  \begin{displaymath}
    \abs[\big]{s_b^{(n)} - s_0^{(n)}} < 2^{-2^n} \abs{z_0}.
  \end{displaymath}
  Let~$\mu$ be the Borchardt mean of~$s$. For~$n\geq 1$, write
  \begin{displaymath}
    q_n = \frac{(u_0^{(n)})^2}{u_0^{(n-1)} \mu},
  \end{displaymath}
  so that for all~$k\geq 0$, we have
  \begin{displaymath}
    \lambda(x,y) = \left(\frac{u_0^{(k)}}{\mu}\right)^{2^k}\prod_{n\geq k} q_{n+1}^{2^n}.
  \end{displaymath}
  These complex numbers~$q_n$ converge quadratically fast to~$1$. To
  be more explicit, we have for all~$n\geq 1$:
  \begin{align*}
    \abs[\big]{u_0^{(n+1)} - v_0^{(n)} t_0^{(n)}}
    &\leq \frac{\sqrt{M}}{2^g} \sum_{b\in\I_g} \left(\abs[\big]{t_b^{(n)}  - t_0^{(n)}}
      + \abs[\big]{v_b^{(n)} - v_0^{(n)}}\right) \\
    &\leq \frac{\sqrt{M}}{2^g \cdot 2\sqrt{m}}
      \sum_{b\in\I_g} \left(\abs[\big]{s_b^{(n)} - s_0^{(n)}}
      + \abs[\big]{u_b^{(n)} - u_0^{(n)}}\right) \\
    &\leq \frac{\sqrt{M}}{2\sqrt{m}}
      \Bigl(2^{-2^n}\abs{z_0} + \frac{1}{2^g}\sum_{b\in \I_g} \abs[\big]{u_b^{(n)} - u_0^{(n)}} \Bigr).
  \end{align*}
  To bound the remaining sum, we write
  \begin{align*}
    \abs[\big]{u_b^{(n)} - u_0^{(n)}}
    &\leq \frac{\sqrt{M}}{2^g}
      \sum_{b'\in\I_g} \abs[\big]{t_{b+b'}^{(n-1)} - t_b^{(n-1)}} \\
    &\leq \frac{\sqrt{M}}{2^g\cdot 2\sqrt{m}} \sum_{b'\in \I_g}\abs[\big]{s_{b+b'}^{(n-1)} - s_b^{(n-1)}}
    \leq \frac{\sqrt{M}}{\sqrt{m}}\, 2^{-2^{n-1}} \abs{z_0}.
  \end{align*}
  Therefore, we have for all~$n\geq 1$
  \begin{displaymath}
    \abs[\big]{u_0^{(n+1)} - v_0^{(n)} t_0^{(n)}}
    \leq \frac{5}{4} \sqrt{\frac{M}{m}}\ 2^{-2^{n-1}}\abs{z_0} =: B\cdot 2^{-2^{n-1}}.
  \end{displaymath}
  We deduce as in~\cite[Thm.~3.10]{labrande_ComputingJacobiTheta2018}
  that
  \begin{displaymath}
    \abs{q_{n+1}-1} \leq B'\cdot 2^{-2^{n-1}}
  \end{displaymath}
  where
  \begin{displaymath}
    B' = 2\abs{z_0} + \frac{1}{m^2}(2MB + B^2) \leq \frac{5M^3}{m^3}.
  \end{displaymath}
  Let~$k\geq 1$ be minimal such that~$B' \cdot 2^{-2^{k-1}}\leq \frac12$. Then we have
  \begin{displaymath}
    \sum_{n\geq k} 2^n\log\abs{q_{n+1}} \leq \sum_{n\geq k}
    2^n \cdot \frac{1}{2}\cdot 2^{2^{k-1}} \cdot 2^{-2^{n-1}} \leq 2^{k}.
  \end{displaymath}
  This proves that the sequence~\eqref{eq:def-ext-borchardt}
  converges; since our estimates are uniform,~$\lambda$ must be
  analytic. Moreover,
  \begin{displaymath}
    \abs{\lambda(x,y)} = \abs[\Big]{\frac{u_0^{(k)}}{\mu}}^{2^{k}}
    \prod_{n\geq k} \abs{q_{n+1}}^{2^n}\leq \exp\left(2^{k}\paren[\big]{1+ \log(M/m)}\right).
  \end{displaymath}
  We obtain the final upper bound on~$\abs{\lambda(x,y)}$ from the
  inequality~$2^{k}\leq 4(1+\log_2(B'))$, after some further
  simplifications. The lower bound comes from the inequality
  \begin{displaymath}
    \sum_{n\geq k} 2^n\log\abs{q_{n+1}} \geq -2^k \cdot 2\log(2)
  \end{displaymath}
  in a similar way.
\end{proof}

\subsection{The general case}
\label{subsec:borchardt-general}

Let~$s$ be a Borchardt sequence containing finitely many bad steps. We
now construct the ``Borchardt mean following~$s$'' in a neighborhood
of the first term of~$s$ as an analytic function, provided that~$s$
contains no zero value. To make things explicit, we introduce the
following quantities:
\begin{itemize}
\item a real number~$M_0>0$ such that $\abs[\big]{s_b^{(0)}} < M_0$ for
  all~$b\in \I_g$;
\item an integer~$n_0$ such that all steps in~$s$ of index~$n\geq n_0$
  are good;
\item a real number~$m_\infty>0$ such
  that~$\paren[\big]{s_b^{(n_0)}}_{b\in \I_g}\in \U_g(m_\infty,M_0)$ in the
  notation of~§\ref{subsec:borchardt-good};
\item for each~$0\leq n\leq n_0-1$, a real number~$m_n>0$ such
  that~$\abs[\big]{s_b^{(n)}} > m_n$ for all~$b\in \I_g$.
\end{itemize}
For each~$n \leq n_0-1$, we also
let~$\smash{\paren[\big]{t_b^{(n)}}_{b\in \I_g}}$ be a collection of
square roots of~$\smash{\paren[\big]{s_b^{(n)}}_{b\in \I_g}}$ such that
the~$n+1$\first term of~$s$ is given by the recurrence
relation~\eqref{eq:borchardt-step}.

It will be useful to introduce Borchardt steps as analytic maps,
besides the case of good sign choices.
Let~$z = (z_b)_{b\in\I_g} \in \C^{2^g}$; assume that~$0 < m < M$ are
real numbers such that~$m < \abs{z_b}^2 < M$ for all~$b$. Then for
each~$b\in \I_g$, there exists a unique analytic square root
map~$\Sqrt_{z_b}$ on the disk~$\D_{m/2}(z_b^2)$ which maps~$z_b^2$
to~$z_b$. Thus, we have a well-defined analytic map
\begin{displaymath}
  \Bstep_z \from \prod_{b\in \I_g} \D_{m/2}(z_b^2) \to \C^{2^g}
\end{displaymath}
A quick calculation shows that~$\nind{d\Bstep_z}\leq \sqrt{(2M+m)/m}$
uniformly on its open set of definition.

\begin{lem}
  \label{lem:borchardt-following-s}
  Given~$s$ and the quantities listed above, let
  \begin{equation}
    \label{eq:borchardt-rho}
    \rho = \min\left\{\frac{m_0}{2}, \frac{m_1}{2}\sqrt{\frac{m_0}{2M_0 + m_0}},\cdots,
      \frac{m_{\infty}}{2} \prod_{j=0}^{n_0-1}\sqrt{\frac{m_j}{2M_0 +m_j}}
    \right\}.
  \end{equation}
  Let~$s^{(0)} = \paren[\big]{s_b^{(0)}}_{b\in \I_g}$ be the first term of~$s$, and
  let~$x\in \D_\rho(s^{(0)})$. Then there exists a unique Borchardt
  sequence~$s'$ with the following properties:
  \begin{enumerate}
  \item the first term of~$s'$ is~$x$;
  \item for all~$0\leq n\leq n_0-1$ and all~$b\in \I_g$, we have
    $\abs[\big]{{s'}_b^{(n)} - s_b^{(n)}} < \frac{1}{2}m_{n}$; moreover
    the~$n+1$\first term of~$s'$ is the result of a Borchardt step with
    choice of square roots~$\Sqrt_{t_b^{(n)}}({s'}_b^{(n)})$ for
    all~$b\in \I_g$;
  \item for all~$n\geq n_0$, the~$n+1$\first term of~$s'$ is the result
    of a Borchardt step from the previous term with good sign choices.
  \end{enumerate}
\end{lem}

\begin{proof}
  We proceed by induction, using the above estimate on derivatives of
  Borchardt steps for~$n\leq n_0-1$.
\end{proof}

\begin{prop}
  \label{prop:bad-borchardt-analytic}
  Given~$s$ and the quantities listed above,
  let~$s^{(0)} = \paren[\big]{s_b^{(0)}}_{b\in \I_g}$ be the first
  term of~$s$, and define~$\rho>0$ as
  in~\eqref{eq:borchardt-rho}. Then there exists a unique analytic
  function~$\mu_s\from \D_\rho(s^{(0)})\to\C$ with the following
  property: for each~$x\in \D_\rho(s^{(0)})$, the value of~$\mu_s$
  at~$x$ is the Borchardt mean of the sequence defined in
  \cref{lem:borchardt-following-s}. We
  have~$\frac12 m_\infty \leq \abs{\mu_s(x)}\leq M_0+\rho$ for
  all~$x\in \D_\rho(s^{(0)})$.
\end{prop}

\begin{proof}
  By \cref{lem:borchardt-following-s}, the function~$\mu_s$ is
  obtained as the composition of a finite number of analytic Borchardt
  steps, followed by an analytic Borchardt mean as defined in
  \cref{prop:good-borchardt-analytic}. The upper bound
  on~$\abs{\mu_s(x)}$ comes from the fact
  that~$\norm{x}\leq M_0+\rho$. For the lower bound, we remark that
  the~$n_0$\th term of the Borchardt sequence of
  \cref{lem:borchardt-following-s} lands
  in~$\U_g(\frac12 m_\infty, M_0+\rho)$.
\end{proof}

We extend this result to the case of extended Borchardt
means. Let~$(u,s)$ be an extended Borchardt sequence containing
finitely many bad steps.  Assume that we are given:
\begin{itemize}
\item a disk~$\D_\rho(z_0)\subset\C$ such
  that~$\rho < \tfrac {1}{17}\abs{z_0}$ (for instance,~$z_0$
  and~$\rho$ may be dyadic);
\item An integer~$n_0$ such that all values in~$s_b^{(n_0)}$ lie
  in~$\D_\rho(z_0)$, and after which all sign choices in~$(u,s)$ are
  good;
\item A real number~$M_0>1$ such that~$\abs[\big]{s_b^{(0)}} < M_0$
  and~$\abs[\big]{u_b^{(0)}} < M_0$ for all~$b\in\I_g$,
  and~$M_0 > \abs{z_0} +\rho$;
\item a real number $0 < m_\infty < 1$ such that the~$n_0$\th term
  of~$u$ lies in~$\U_g(m_\infty,M_0)$,
  and~$m_\infty < \abs{z_0}-\rho$;
\item For each~$0\leq n\leq n_0-1$, a real number $m_n >0$ such
  that~$\abs[\big]{s_b^{(n)}} > m_n$ and~$\abs[\big]{u_b^{(n)}} > m_n$
  for all~$b\in\I_g$.
\end{itemize}
For each~$n \leq n_0-1$, we also
let~$\paren[\big]{t_b^{(n)}}_{\smash{b\in \I_g}}$
and~$\paren[\big]{v_b^{(n)}}_{\smash{b\in \I_g}}$ be collections of
square roots of~$\paren[\big]{s_b^{(n)}}_{b\in \I_g}$
and~$\paren[\big]{u_b^{(n)}}_{b\in \I_g}$ respectively such that
the~$n+1$\first term of~$(u,s)$ is given by the recurrence
relation~\eqref{eq:ext-borchardt-step}.

The following lemma and proposition are proved by the same methods we
used for regular Borchardt means, and we omit their proofs.

\begin{lem}
  \label{lem:ext-borchardt-following-s}
  Given~$(u,s)$ and the quantities listed above, let
  \begin{equation}
    \label{eq:ext-borchardt-rho}
    \rho = \min_{0\leq n\leq n_0}\left(\frac{m_n}{2}
      \prod_{j=0}^{n-1} \sqrt{\frac{m_n}{2M_0+m_n}}\right),
  \end{equation}
  with the convention that~$m_{n_0}= m_\infty$.  Let~$(u^{(0)},s^{(0)})$ be
  the first term of~$(u,s)$, and
  let~$(x,y)\in \D_\rho\paren[\big]{(u^{(0)},s^{(0)})}$. Then there exist
  extended Borchardt sequences~$(u',s')$ with the following
  properties:
  \begin{enumerate}
  \item the first term of~$(u',s')$ is~$(x,y)$;
  \item for each~$0\leq n\leq n_0-1$ and each~$b\in \I_g$, we
    have
    \begin{displaymath}
      \abs[\big]{{s'}_b^{(n)} - s_b^{(n)}} < \frac12 m_n
      \quad\text{and}\quad
      \abs[big]{{u'}_b^{(n)} - u_b^{(n)}} < \frac12 m_n;
    \end{displaymath}
    moreover the~$n+1$\first term of~$(u',s')$ is the result of an
    extended Borchardt step with choices of square
    roots~$\Sqrt_{t_b^{(n)}}({s'}_b^{(n)})$
    and~$\Sqrt_{v_b^{(n)}}({u'}_b^{(n)})$ for all~$b\in \I_g$;
  \item for all~$n \geq n_0$, the~$n+1$\first term of~$(u,s)$ is obtained
    from the previous one by an extended Borchadt step with good sign
    choices.
  \end{enumerate}
  These extended Borchardt sequences coincide up to their~$n_0$\th
  terms, and their extended Borchardt means are equal.
\end{lem}

\begin{prop}
  \label{prop:bad-ext-borchardt-analytic}
  Given~$(u,s)$ and the quantities listed above,
  let~$(u^{(0)},s^{(0)})$ be the first term of~$(u,s)$, and
  define~$\rho>0$ as in~\eqref{eq:ext-borchardt-rho}.  Then there
  exists a unique analytic
  function~$\lambda_{(u,s)}\from \D_\rho(z_0)\to\C$ with the following
  property: for each~$(x,y)\in \D_\rho(z)$, the value
  of~$\lambda_{(u,s)}$ at~$x$ is the extended Borchardt mean of any of
  the extended Borchardt sequences defined in
  \cref{lem:ext-borchardt-following-s}. Moreover, we have
  \begin{displaymath}
    \exp\paren[\big]{-28\log^2(4M/m)} \leq
    \abs{\lambda_{u,s}(x,y)} \leq \exp\paren[\big]{20\log^2(4M/m)}
  \end{displaymath}
  where~$m = \frac12 m_\infty$ and~$M = M_0+\rho$.
\end{prop}

\begin{rem}
  In~\cite[§6.1]{dupont_FastEvaluationModular2011},
  \cite[§7.4.2]{dupont_MoyenneArithmeticogeometriqueSuites2006},
  \cite[§3.4]{labrande_ComputingJacobiTheta2018},
  and~\cite[Prop.~3.7]{labrande_ComputingThetaFunctions2016} it is
  shown that the analytic functions~$\mu,\lambda,\mu_s$
  and~$\lambda_{(u,s)}$ that we just defined can be evaluated at any
  given complex point in quasi-linear
  time~$O\paren[\big]{\M(N)\log N}$ in the required precision,
  where~$\M(N)$ denotes the cost of multiplying $N$-bit
  integers. In fact, these proofs show that these analytic functions
  can be evaluated in \emph{uniform} quasi-linear time. In the case
  of~$\mu_s$ and~$\lambda_{(u,s)}$, the implied constant only depends
  on the auxiliary data listed in this section, not on the Borchardt
  sequences themselves.
\end{rem}

\section{Newton schemes for theta functions}
\label{sec:dupont}

In this section, we present the different Newton schemes used for the
computation of theta constants and theta functions in genus~$1$
and~$2$ as well as possible extensions to higher genera,
following~\cite{dupont_FastEvaluationModular2011,
  dupont_MoyenneArithmeticogeometriqueSuites2006,
  labrande_ComputingJacobiTheta2018,
  labrande_ComputingThetaFunctions2016}. We formulate them in terms of
the analytic Borchardt functions introduced
in~§\ref{sec:borchardt}. In the three cases of theta functions in
genus~$1$ and theta constants in genus~$1$ and~$2$, we are able to
write down the inverse of the analytic function~$\C^r\to \C^r$ used in
the Newton scheme in an explicit way. This provides us with all the
necessary data to apply the results of~§\ref{sec:newton} and obtain
explicit convergence results for these Newton schemes.

\subsection{General picture}
\label{subsec:general-picture}

The Newton schemes we consider to compute theta constants at a given
point~$\tau\in \Half_g$ use increasingly better approximations of the
point
\begin{equation}
  \label{eq:theta-quotients}
  \Theta(\tau) = 
  \left(\frac{\theta_{0,b}(0,\tau/2)}{\theta_{0,0}(0,\tau/2)}\right)_{b\in \I_g\setminus\set{0}}
  \in \C^{2^g-1}.
\end{equation}
From this input, computing certain Borchardt means will provide
approximations of the quantities~$\theta^2_{0,b}(0,N\tau)$, for any
symplectic matrix~$N\in \Sp_{2g}(\Z)$ that we might choose. Recall
that a matrix~$N\in \Sp_{2g}(\Z)$ with~$g\times g$
blocks~$ \left(\begin{smallmatrix} a&b\\c&d
\end{smallmatrix}\right)$ acts on~$\Half_g$
as $N\tau = (a\tau+b)(c\tau+d)^{-1}$, and on~$\C^g\times \Half_g$
as~$N\cdot(z,\tau) = \paren[\big]{(c\tau+d)^{-t}z, N\tau}$,
where~${}^{-t}$ denotes inverse transposition. The next proposition,
derived from the works mentioned above, is key.

\begin{prop}
  \label{prop:theta-as-borchardt}
  Let~$\tau\in \Half_g,$ let~$z\in \C^g$, and let
  $\lambda,\mu\in \C^\times$. Then
  \begin{enumerate}
  \item The sequence
    \label{it:lambda}
    \begin{equation}
      \label{eq:theta-as-borchardt-4}
      \left(\frac{\theta_{0,b}^2(0,2^n\tau)}{\mu}\right)_{b\in \I_g,n\geq 0}
    \end{equation}
    is a Borchardt sequence with Borchardt mean~$1/\mu$, obtained
    from the choice of square roots
    \begin{displaymath}
      \left(\frac{\theta_{0,b}(0,2^n\tau)}{\sqrt{\mu}}\right)_{b\in \I_g}
    \end{displaymath}
    for some choice of~$\sqrt{\mu}$, at each step.
  \item All sequences of the form
    \label{it:mu}
    \begin{equation}
      \label{eq:theta-as-ext-borchardt}
      \left(\frac{\theta_{0,b}^2(z, 2^n\tau)}{\lambda^{2^{-n}}\mu^{1-2^{-n}}},
        \frac{\theta_{0,b}^2(0,2^n\tau)}{\mu}\right)_{b\in \I_g,n\geq 0}
    \end{equation}
    with compatible choices of~$2^{-n}$-th roots
    (i.e.~such that $(\lambda^{2^{-n-1}})^2 = \lambda^{2^{-n}}$
    and $(\mu^{1-2^{-n-1}})^2 = \mu\cdot \mu^{1-2^{-n}}$
    for all~$n$) are extended Borchardt sequences with extended
    Borchardt mean~$1/\lambda$; they precisely are the sequences obtained
    from choices of square roots of the form
    \begin{displaymath}
      \left(\frac{\theta_{0,b}(z,
          2^n\tau)}{\lambda^{2^{-n-1}}\mu^{(1-2^{-n})/2}},
        \frac{\theta_{0,b}(0,2^n\tau)}{\sqrt{\mu}}\right)_{b\in
        \I_g}
    \end{displaymath}
    for some choice of square roots of~$\mu, \lambda^{2^{-n}}$
    and~$\mu^{1-2^{-n}}$, at each step.
  \end{enumerate}
\end{prop}

Consider first the case of theta constants. From the theta
quotients~\eqref{eq:theta-quotients}, one can compute all squared
theta quotients of the
form~$\theta_{a,b}^2(0,\tau)/\theta^2_{0,0}(0,\tau/2)$ using the
duplication formula. Then, applying the transformation formulas
under~$\Sp_{2g}(\Z)$~\cite[§II.5]{mumford_TataLecturesTheta1983}
allows us to compute all theta quotients of the
form~$\theta_{0,b}^2(0, N\tau)/\theta_{0,0}^2(0, N\tau)$
for~$b\in \I_g$. Finally, applying
\cref{prop:theta-as-borchardt},~\eqref{it:lambda} gives us access
to~$\mu = \theta_{0,0}^2(0,N\tau)$, so that we can recover
all~$\theta_{0,b}^2(0,N\tau)$, as promised. At the end of the
algorithm, we apply the transformation formulas once more: the
relations between squared theta values~$\theta_{a,b}^2(0,\tau)$
and~$\theta_{a,b'}^2(0,N\tau)$ involve a factor~$\det(C\tau+D)$
where~$C,D$ are the lower~$g\times g$ blocks of~$N$. These
determinants are simple functions of the entries of~$\tau$, and we use
this feedback in a Newton scheme to compute a better approximation of
the initial theta quotients~\eqref{eq:theta-quotients}.  When an
appropriate precision is reached, we repeat the above process one last
time to return approximations of the squared theta
values~$\theta_{a,b}^2(0,\tau)$.

In the case of theta functions, we consider the following larger set
of theta quotients:
\begin{equation}
  \label{eq:theta-quotients-2}
  \Theta'(\tau) = 
  \left(\frac{\theta_{0,b}(0,\tau/2)}{\theta_{0,0}(0,\tau/2)},
    \frac{\theta_{0,b}(z,\tau/2)}{\theta_{0,b}(z,\tau/2)}
  \right)_{b\in \I_g\setminus\set{0}}
  \in \C^{2^{g+1}-2}.
\end{equation}
We obtain the theta
quotients~$\theta^2_{0,b}\paren[\big]{N\cdot(z,\tau)}/\theta_{0,0}^2\paren[\big]{N\cdot(z,\tau)}$
from the transformation formulas, and
\cref{prop:theta-as-borchardt},~\eqref{it:mu} allows us to compute
$\lambda = \theta_{0,0}^2\paren[\big]{N\cdot(z,\tau)}$. The feedback
is again provided by transformation formulas, and involves simple
functions (determinants and exponentials) in the entries of~$z$
and~$\tau$.

In order to run this algorithm, one has to make the correct choices of
square roots each time \cref{prop:theta-as-borchardt} is applied. At
the end of the loop, when using feedback on~$z$ and~$\tau$ to obtain
theta values at a higher precision, one \emph{assumes} that the
Jacobian matrix of the system is well-defined and invertible; in
particular, it must be a square matrix. In practice, one computes an
approximation of this Jacobian matrix using finite differences; the
resulting Newton scheme is of the type studied in~§\ref{sec:newton}.

We close this presentation with a discussion on argument
reduction. Before attempting to run these Newton schemes, one should
\emph{reduce} the input~$(z,\tau)$ using symmetries of theta
functions. Performing this reduction is necessary to even hope for
algorithms with uniform complexities in~$(z,\tau)$.  If~$g\leq 2$, it
is possible to use the action of~$\Sp_{2g}(\Z)$ on~$\tau$ to reduce it
to the \emph{Siegel fundamental domain}~$\Fund_g\subset\Half_g$
defined by the following
conditions~\cite[§I.3]{klingen_IntroductoryLecturesSiegel1990}:
\begin{itemize}
\item $\im(\tau)$ is Minkowski-reduced;
\item $\abs{\re(\tau_{i,j})}\leq 1/2$ for all~$1\leq i,j\leq g$;
\item $\abs{\det(C\tau+D)}\geq 1$ for all~$g\times g$ matrices~$C,D$
  forming the lower blocks of a symplectic matrix~$N\in \Sp_{2g}(\Z)$;
  in particular we have~$\abs{\tau_{i,i}}\geq 1$ for
  all~$1\leq i\leq g$, so that~$\im(\tau_{i,i})\geq \sqrt{3}/2$.
\end{itemize}
The reduction algorithm is described
in~\cite[§6]{streng_ComputingIgusaClass2014}.

In fact, it is possible to obtain useful information on values of
theta functions, and to study the Newton schemes described above,
without assuming that all the conditions defining~$\Fund_g$ hold: see
for instance~\cite[Prop.~7.6]{streng_ComputingIgusaClass2014}. On the
other hand, we will additionally assume that the imaginary part
of~$\tau$ is bounded; this assumption is necessary to show that the
Newton schemes converge uniformly.  Other inputs can be handled using
duplication formulas and the naive algorithm:
see~\cite[§6.3]{dupont_FastEvaluationModular2011}
and~\cite[§4.2]{labrande_ComputingJacobiTheta2018} in the genus~1
case. We will adapt this strategy to obtain a uniform algorithm for
genus~2 theta constants in~§\ref{sec:unif}.

The argument~$z$ can be reduced as well. By periodicity of the theta
function~$\theta(\cdot,\tau)$ with respect to the
lattice~$\Z^g + \tau\Z^g$~\cite[§II.1]{mumford_TataLecturesTheta1983},
it is always possible to assume that~$\abs{\re(z_i)}\leq \frac12$ for
each~$i$, and that
\begin{displaymath}
  \left(
    \begin{matrix}
      \im(z_1)\\ \vdots\\ \im(z_g)
    \end{matrix}
  \right)
  = \im(\tau)
  \left(
    \begin{matrix}
      v_1\\ \vdots\\ v_g
    \end{matrix}
  \right)
\end{displaymath}
for some vector~$v\in \R^g$ such that~$\abs{v_i}\leq \frac12$ for
all~$i$. Since duplication formulas relate the values of theta
functions at~$z$ and~$2z$, we can in fact assume that~$z$ is very
close to zero, for instance~$\abs{z_i} < 2^{-n}$ for some fixed~$n$.

In the rest of this section, we analyze the Newton systems more
closely in the case of theta constants of genus~$1$ and~$2$, as well
as theta functions in genus~$1$, for suitably reduced inputs; our goal
is to apply \cref{thm:practical-newton}. We also discuss the situation
in higher genera.

\subsection{Genus~1 theta constants}
\label{subsec:g1-cst}

In the case of genus~$1$ theta constants, the Newton system is
univariate, and~$\tau$ is simply a complex number with positive
imaginary part. Let~$\Red_1\subset \Half_1$ be the compact set defined
by the following conditions:
\begin{itemize}
\item $\abs{\re(\tau)}\leq \frac12$;
\item $\abs{\tau}\geq 1$;
\item $\im(\tau)\leq 2$.
\end{itemize}
Thus,~$\Red_1$ is a truncated, closed version of the usual fundamental
domain~$\Fund_1$. The only matrix in~$\Sp_{2}(\Z) = \SL_2(\Z)$ that we
consider is
\begin{displaymath}
  N = \mat{0}{-1}{1}{0},
\end{displaymath}
so that~$N\tau =
-1/\tau$. By~\cite[§I.7]{mumford_TataLecturesTheta1983}, we have
\begin{equation}
  \label{eq:feedback-g1-cst}
  \theta_{0,0}^2(0,N\tau) = -i\tau \theta_{0,0}^2(0,\tau).
\end{equation}
It turns out that the two Borchardt sequences used in the algorithm,
namely
\begin{equation}
  \label{eq:borchardt-g1-cst}
  s_1 = \paren*{\frac{\theta_{0,b}^2(0,2^n\tau)}{\theta_{0,0}^2(0,\tau)}}_{b\in \Z/2\Z, n\geq 0}
  \quad \text{and}
  \quad
  s_2 = \paren*{\frac{\theta_{0,b}^2(0,2^n N\tau)}{\theta_{0,0}^2(0,N\tau)}}_{b\in \Z/2\Z, n\geq 0}
\end{equation}
are given by good sign choices only:
see~\cite[Thm.~2]{dupont_FastEvaluationModular2011}
and~\cite[Lem.~2.9]{cox_ArithmeticgeometricMeanGauss1984}. Our first
aim is to collect the data listed in~§\ref{subsec:borchardt-good} for
these sequences. This can be done by looking at the theta
series~\eqref{eq:theta-function} directly; see for
instance~\cite[Lem.~3.3]{labrande_ComputingJacobiTheta2018}. We
formulate the following result in the more general context of theta
functions, since it will also be useful in~§\ref{subsec:g1-fun}.

\begin{lem}
  \label{lem:theta-inequality-g1}
  Let~$(z,\tau)\in \C\times\Half_1$ be such that
  that~$\abs{\im(z)} < 2\im(\tau)$, and write $q =
  \exp(-\pi\im(\tau))$. Then we have
  \begin{displaymath}
    \abs[\big]{\theta_{0,b}(z,\tau)-1} < 2q\cosh(2\pi\im(z))
    + \frac{2q^4\exp(4\pi\abs{\im(z)})}
    {1 - q^5\exp(2\pi\abs{\im(z)})}
  \end{displaymath}
  for all~$b\in \Z/2\Z$, and
  \begin{displaymath}
    \abs[\bigg]{\frac{\theta_{1,0}(z,\tau)}{\exp(\pi i\tau/4)}
      - \paren[\big]{\exp(\pi i z) + \exp(-\pi i z)}}
    < \frac{2q^2 \exp(3\pi \abs{\im z})}{1 - q^4 \exp(2\pi \abs{\im z})}.
  \end{displaymath}
 
\end{lem}

\begin{proof}
  For the first inequality, write
  \begin{displaymath}
    \theta_{0,b}(z,\tau) = 1 + \exp(\pi i\tau + 2\pi iz) + \exp(\pi i\tau - 2\pi iz)
    + \sum_{n\in \Z,\, \abs{n}\geq 2} \exp(\pi i n^2\tau + 2\pi i n z).
  \end{displaymath}
  The modulus of this last sum can be bounded above by
  \begin{equation}
    \label{eq:theta-tail-g1}
    2\sum_{n\geq 2} \exp\paren[\big]{-\pi n^2 \im(\tau) + 2\pi n \abs{\im z}},
  \end{equation}
  and we conclude by comparing~\eqref{eq:theta-tail-g1} with the sum
  of a geometric series matching its first two terms. The proof of the
  second inequality is similar and omitted.
\end{proof}



In particular, for each~$\tau\in \Red_1$, we
have~$\abs[\big]{\theta_{0,0}(0,\tau/2) - 1} < 0.53$, so
that~$\theta_{0,0}(0,\tau/2)$, which appears as the denominator
of~$\Theta(\tau)$ in~\eqref{eq:theta-quotients}, is indeed
nonzero. More numerical computations will appear in subsequent proofs;
we will only write down the first few digits of all real numbers
involved.

\begin{prop}
  \label{lem:theta-g1-cst}
  Let~$\tau\in \Red_{1}$. Then, in the notation
  of~\emph{§\ref{subsec:borchardt-good}}, the following bounds apply
  to the Borchardt sequence~\eqref{eq:borchardt-g1-cst}
  with~$\mu = \theta_{0,0}^2(0,\tau)$:
  \begin{displaymath}
    m_0 = 0.56 \quad\text{and}\quad M_0 = 1.7.
  \end{displaymath}
  The following bounds apply to the
  sequence~\eqref{eq:borchardt-g1-cst} taken at~$N\tau$
  with~$\mu = \theta_{0,0}^2(0,N\tau)$:
  \begin{displaymath}
    m_0 = 0.13 \quad\text{and}\quad M_0 = 1.38.
  \end{displaymath}
\end{prop}

\begin{proof}
  For the first sequence, we note that $\exp(-\pi \im(\tau)) < 0.066$,
  and conclude using \cref{lem:theta-inequality-g1}. For the second
  sequence, we invoke the transformation formula: we have
  \begin{displaymath}
    \frac{\theta_{0,1}(0,-1/\tau)}{\theta_{0,0}(0,-1/\tau)}
    = \frac{\theta_{1,0}(0,\tau)}{\theta_{0,0}(0,\tau)}.
  \end{displaymath}
  By \cref{lem:theta-inequality-g1}, the angle
  between~$\theta_{0,0}(0,\tau)$ and~$\theta_{1,0}(0,\tau)$ seen from
  the origin is at most~$0.95 < \pi/2$; moreover we have
  $0.41 < \abs{2\exp(i\pi\tau/4)} < 1.02$, from which 
  the claimed bounds follow.
\end{proof}

\begin{thm}
  \label{thm:analytic-g1-cst}
  Let~$\rho = 1.4\cdot 10^{-4}$, define~$\Theta$ as
  in~\eqref{eq:theta-quotients} for~$g=1$, and let
  \begin{displaymath}
    \V = \bigcup_{\tau\in \Red_1} \D_\rho\paren[\big]{\Theta(\tau)}.
  \end{displaymath}
  Then the operations described in~\emph{§\ref{subsec:general-picture}},
  taking good choices of square roots always, combined
  with~eq.~\eqref{eq:feedback-g1-cst} define an analytic
  function~$F\from \V\to \C$ such that
  \begin{displaymath}
    F\paren[\big]{\Theta(\tau)}
    = \tau
  \end{displaymath}
  for each~$\tau\in \Red_1$. We have~$\abs{F(x)}\leq 27$ for all~$x\in\V$.
\end{thm}

\begin{proof}
  We backtrack from the result of the previous
  proposition. Let~$\tau\in \Red_1$. Then the Borchardt means we take
  are well-defined as analytic functions on any open set where the
  theta quotients
  \begin{equation}
    \label{eq:first-terms-g1-cst}
    \frac{\theta_{0,1}^2(\tau)}{\theta_{0,0}^2(\tau)}
    \quad\text{and}\quad
    \frac{\theta_{1,0}^2(\tau)}{\theta_{0,0}^2(\tau)}
  \end{equation}
  are perturbed by a complex number of modulus at most~$m = 0.13$. We
  construct~$\V$ in such a way that the maximal perturbation will not
  exceed~$m/2$. The quantities~\eqref{eq:first-terms-g1-cst} are
  obtained as quotients of the form:
  \begin{equation}
    \label{eq:intermediate-g1-cst}
    \frac{\theta_{a,b}^2(\tau)/\theta_{0,0}^2(\tau/2)}{\theta_{0,0}^2(\tau)/\theta_{0,0}^2(\tau/2)}.
  \end{equation}
  By \cref{lem:theta-inequality-g1}, the modulus of the denominator is
  at least~$0.32$, the modulus of the numerator is at
  most~$5.7$. Hence, each of the individual theta
  quotients~\eqref{eq:intermediate-g1-cst} may be perturbed by any
  complex number of modulus at most~$6.2\cdot 10^{-4}$. In turn, these
  quotients are obtained from the duplication formula applied to~$1$
  and~$\Theta(\tau)$; the modulus of these two complex numbers are at
  most~$2.13$, hence they may be perturbed
  by~$\rho = 1.4\cdot 10^{-4}$. By construction, the value taken by
  the resulting Borchardt means at any~$x\in \V$ has modulus at
  least~$0.066$ and at most~$1.8$, hence the final bound
  on~$\abs{F(x)}$.
\end{proof}

Since the inverse of~$F$ is given by theta constants, we easily see
that the Jacobian of~$F$ is invertible at all the relevant points, in
a uniform way.

\begin{prop}
  \label{prop:inverse-g1-cst}
  For each~$\tau\in \Red_1$, we have
  \begin{displaymath}
    \nind{d\Theta(\tau)} \leq 125.
  \end{displaymath}
\end{prop}

\begin{proof}
  By \cref{lem:theta-inequality-g1}, the denominator of this function
  has modulus at least~$0.47$, and its numerator has modulus at
  most~$1.53$. The result will then follow from an upper bound on the
  quantities~$\norm{d\theta_{0,b}(\tau/2)}$. We can derive such bounds
  from \cref{prop:cauchy}, noting that~$\theta_{0,b}$ is an analytic
  function defined on~$\D_{1/4}(\tau/2)$, and has modulus at
  most~$2.34$ on this disk by \cref{lem:theta-inequality-g1}.
\end{proof}

Combining \cref{thm:analytic-g1-cst,prop:inverse-g1-cst} with the
results of~§\ref{sec:newton}, we obtain:

\begin{cor}
  For all~$\tau\in \Red_1$, the Newton scheme described
  in~\emph{§\ref{subsec:general-picture}} to compute theta constants
  at~$\tau$ will converge starting from approximations
  of~$\Theta(\tau)$ to~$60$ bits of precision.
\end{cor}

\subsection{Genus~1 theta functions}
\label{subsec:g1-fun}

In the case of genus~$1$ theta functions, Newton iterations are
performed using two complex variables. As in~§\ref{subsec:g1-cst}, we
only use the symplectic matrix
\begin{displaymath}
  N = \mat{0}{-1}{1}{0}.
\end{displaymath}
The Newton scheme involves the extended Borchardt mean of the
sequence~\eqref{eq:theta-as-ext-borchardt} with
$\lambda = \theta_{0,0}^2(z,\tau)$ and~$\mu = \theta_{0,0}^2(0,\tau)$,
as well as the analogous sequence taken at
$N\cdot(z,\tau) = (z/\tau, -1/\tau)$ instead of~$(z,\tau)$. Feedback
is then provided by the two following
equalities~\cite[§I.7]{mumford_TataLecturesTheta1983}:
\begin{equation}
  \begin{aligned}
    \label{eq:feedback-g1-fun}
    \theta_{0,0}^2(0, N\tau) &= -i\tau \theta_{0,0}^2(0,\tau), \\
    \theta_{0,0}^2\paren[\big]{N\cdot(z,\tau)} &= -i\tau \exp(2\pi i
    z^2/\tau) \theta_{0,0}^2(z,\tau).
  \end{aligned}
\end{equation}
Both extended Borchardt sequences are given by good sign choices
only~\cite[Prop.~4.1]{labrande_ComputingJacobiTheta2018}, provided
that the following reductions are met: $\tau\in \Red_{1}$, and the
inequalities
\begin{equation}
  \label{eq:g1-z-reduction}
  \abs{\im (z)} \leq \frac18 \im(\tau), \quad \abs{\re(z)} \leq \frac18
\end{equation}
are satisfied. Let~$\S_1\subset \C\times\Half_1$ be the compact set of
such~$(z,\tau)$. Our first goal is to collect the necessary data to apply
\cref{prop:bad-ext-borchardt-analytic} in this context.

\begin{lem}
  \label{lem:theta-inequalities-g1-fun}
  Let~$(z,\tau)\in \S_{1}$, and let~$a,b\in \Z/2\Z$. Then the
  following inequalities hold:
  \begin{align*}
    \abs{\theta_{0,b}(z,\tau/2)-1} &< 0.68, \\
    \abs{\theta_{a,b}(z,\tau)} &< 1.17,\quad\text{and}\\    
    \abs{\theta_{1,0}(z, \tau)} &> 0.37.
  \end{align*}
\end{lem}

\begin{proof}
  These inequalities are direct consequences of
  \cref{lem:theta-inequality-g1}. Let us only detail the lower bound
  on~$\abs{\theta_{1,0}(z,\tau)}$. Since~$\abs{\re z} \leq \frac18$,
  we have
  \begin{displaymath}
    \re(\exp(\pi i z) + \exp(-\pi i z)) \geq 2\cos(\pi/8) \cosh(\pi\im(z)) > 1.84.
  \end{displaymath}
  Therefore,
  \begin{displaymath}
    \abs{\theta_{1,0}(z,\tau)} > \exp(-\pi \im(\tau)/4) (1.84 - 0.025) > 0.37.\qedhere
  \end{displaymath}
\end{proof}

\begin{prop}
  \label{prop:data-g1-fun}
  Let~$(z,\tau)\in \S_{1}$. Then, in the notation
  of~\emph{§\ref{subsec:borchardt-general}}, the following bounds
  apply to the extended Borchardt
  sequence~\eqref{eq:theta-as-ext-borchardt}, where
  $\lambda = \theta_{0,0}^2(z,\tau)$ and
  $\mu = \theta_{0,0}^2(0,\tau)$:
  \begin{displaymath}
    n_0=1,\quad M_0 = 1.94, \quad m_0 =0.51 \quad\text{and} \quad m_\infty = 0.72.
  \end{displaymath}
  The following bounds apply to the extended Borchardt
  sequence~\eqref{eq:theta-as-ext-borchardt} taken
  at~$N\cdot(z,\tau)$,
  with~$\lambda = \theta_{0,0}^2\paren[\big]{N\cdot(z,\tau)}$
  and~$\mu = \theta_{0,0}^2(0,N\tau)$:
  \begin{displaymath}
    n_0=1,\quad M_0 = 1.69,\quad m_0 = 0.1, \quad \text{and} \quad
    m_\infty = 0.51.
  \end{displaymath}
\end{prop}

\begin{proof}
  These explicit values are also derived from
  \cref{lem:theta-inequality-g1}. In the case of the second Borchardt
  sequence, we analyze the first term using the transformation formula
  for theta functions under~$\SL_2(\Z)$. For the next terms, we use
  the following inequalities:
  \begin{displaymath}
    \im(-1/\tau) = \frac{\im(\tau)}{\abs{\tau}^2}
    \geq \frac{\sqrt{\abs{\tau}^2 - \frac12}}{\abs{\tau}^2} \geq 0.46,
  \end{displaymath}
  \begin{displaymath}
    \im(z/\tau) = \frac{1}{\abs{\tau}^2}
    \abs[\big]{\im(z)\re(\tau) - \re(z)\im(\tau)}
    \leq \frac{3}{16}\im(-1/\tau),
  \end{displaymath}
  so that for instance
  \begin{displaymath}
    \abs[\big]{\theta_{0,0}^2(z/\tau, -1/\tau)} < 1.78.\qedhere
  \end{displaymath}
\end{proof}

\begin{thm}
  \label{thm:analytic-g1-fun}
  Let~$\rho = 2.9\cdot 10^{-5}$, define~$\Theta'$ as
  in~\eqref{eq:theta-quotients-2} for~$g=1$, and let
  \begin{displaymath}
    \V = \bigcup_{(z,\tau)\in \S_1} \D_\rho \paren[\big]{\Theta'(\tau)}.
  \end{displaymath}
  Then the operations described
  in~\emph{§\ref{subsec:general-picture}}, taking good choices of
  square roots always, combined with the
  formulas~\eqref{eq:feedback-g1-fun} define an analytic
  function~$F\from \V\to \C^2$ such that
  \begin{displaymath}
    F(\Theta'(\tau)) = \paren*{\tau, \exp(2\pi i z^2/\tau)}
  \end{displaymath}
  for each~$(z,\tau)\in \S_1$. We
  have~$\norm{F(x)}\leq 4.3\cdot 10^{221}$ uniformly on~$\V$.
\end{thm}

\begin{proof}
  We apply \cref{prop:bad-ext-borchardt-analytic} with the explicit
  values provided above; to find an acceptable~$\rho$, we follow a
  backtracking strategy as in the proof of \cref{thm:analytic-g1-cst},
  using the first two inequalities of
  \cref{lem:theta-inequalities-g1-fun}.  The upper bound on~$\norm{F}$
  comes from \cref{prop:bad-ext-borchardt-analytic}.
\end{proof}

The upper bound on~$\norm{F}$ could certainly be improved in this
situation, but the above value will be sufficient for our purposes.
  
This function~$F$ admits an analytic reciprocal. Here it is essential
that the theta constants~$\theta_{0,b}(z,\tau)$ for~$b\in \Z/2\Z$ are
invariant under~$z\mapsto -z$; this implies that they can be rewritten
as analytic functions of~$z^2$.

\begin{prop}
  \label{prop:theta-from-square}
  Let~$b\in \Z/2\Z$. Then there exists a unique analytic function
  \begin{displaymath}
    \zeta_{0,b}\from \C\times\Half_1\to \C
  \end{displaymath}
  such that for all~$(z,\tau) \in\C\times\Half_1$, we
  have~$\theta_{0,b}(z,\tau) = \zeta_{0,b}(z^2,\tau)$.
\end{prop}

\begin{proof}
  Consider the following reorganization of the theta series:
  \begin{displaymath}
    \theta_{0,b}(z,\tau) =
    1 + \sum_{n\geq 1} (-1)^{nb}\exp(\pi i n^2\tau) (\exp(2\pi i n z) + \exp(-2\pi i n z)).
  \end{displaymath}
  Each factor~$\exp(2\pi i nz) + \exp(-2\pi i nz)$, as an even entire
  function, has only powers of~$z^2$ in its Taylor series. We obtain a
  candidate~$\zeta_{0,b}$ as a formal power series, easily seen to
  converge uniformly on compact sets of~$\C\times\Half_1$.
\end{proof}

Note that for every~$(z,\tau)\in \S_1$, we
have~$\re(z^2/\tau) < 1/2$; therefore~$\exp(2\pi i z^2/\tau)$ lands in
the domain of definition of the principal branch of the complex
logarithm, denoted by~$\U$. Consider the two following maps:
\begin{displaymath}
  \begin{matrix}
    \Half_1 \times \U &\to &\C\times \Half_1&\to &\C\times\C\\[6pt]
    (\tau, x) &\mapsto &\paren[\big]{\frac{1}{2\pi i}\log(x), \tau)} \\
    & & (y, \tau) &\mapsto & \displaystyle\paren*{\frac{\theta_{0,1}(0,\tau/2)}
      {\theta_{0,0}(0,\tau/2)},
    \frac{\zeta_{0,1}(y,\tau/2)}{\zeta_{0,0}(y,\tau/2)}}.
  \end{matrix}  
\end{displaymath}
Call~$G$ their composition; it is well-defined on an open neighborhood
of the image of~$\S_1$
by~$(z,\tau)\mapsto \paren[\big]{\tau,\exp(2\pi i z^2/\tau)}$, and is
the reciprocal of~$F$.


\begin{prop}
  \label{prop:inverse-g1-fun}
  For each~$(z,\tau) \in \S_1$, we have
  \begin{displaymath}
    \nind[\big]{dG\paren[\big]{\tau, \exp(2\pi i z^2/\tau)}} \leq 8.6\cdot 10^4.
  \end{displaymath}
\end{prop}

\begin{proof}
  Let~$x = \exp(2\pi i z^2/\tau)$, and~$y = z^2$. We have
  \begin{displaymath}
    \abs[\big]{\im(z^2/\tau)} = \frac{1}{\abs{\tau}^2}
    \cdot\frac{1}{64}\paren[\big]{\im(\tau) + \im(\tau)^3 + \im(\tau)}
    \leq \frac{1}{16},
  \end{displaymath}
  showing that~$\abs{x}$ is close to~$1$. It only remains to obtain
  explicit upper bounds on the derivative
  of~$\zeta_{0,1}/\zeta_{0,0}$. To obtain such bounds, we consider the
  polydisk of radius~$1/16$ centered in~$(y,\tau/2)$;
  by~\cref{lem:theta-inequality-g1}, we
  have~$\abs{\zeta_{0,b}} < 5.8$ on this disk for each~$b\in \Z/2\Z$,
  so that~$\nind{d\zeta_{0,b}(y,\tau/2)} < 277$ by \cref{prop:cauchy}. We
  can conclude using the lower bound on~$\abs{\zeta_{0,0}(y,\tau/2)}$
  provided by \cref{lem:theta-inequalities-g1-fun}.
\end{proof}

\begin{cor}
  For all~$(z,\tau)\in \S_1$, the Newton scheme described
  in~\emph{§\ref{subsec:general-picture}} to compute theta functions
  at~$(z,\tau)$ will converge starting from approximations
  of~$G(\tau)$ to~$1600$ bits of precision.
\end{cor}

\subsection{Genus~2 theta constants}
\label{subsec:g2}

In the case of genus~$2$ theta constants, Newton iterations are
performed on three variables, and feedback is provided by the action
of three symplectic matrices.

For general~$g$, certain interesting symplectic matrices can be
written down explicitly. Denote the elementary~$g\times g$ matrices
by~$E_{i,j}$ for~$1\leq i,j\leq g$, and let~$I$ be the identity
matrix. Let~$M_i$ and~$N_{i,j}$ ($i\neq j$) be the following
symplectic matrices, written in~$g\times g$ blocks:
\begin{displaymath}
  M_i = \mat{-I}{-E_i}{E_i}{-I + E_i}, \quad
  N_{i,j} = \mat{-I}{-E_{i,j}-E_{j,i}}{E_{i,j}+E_{j,i}}{-I + E_{i,i} + E_{j,j}}
\end{displaymath}
The matrices~$M_i$ and~$N_{i,j}$ are precisely engineered so that
determinants of the form~$\det(C\tau+D)$ give us direct access to the
entries of~$\tau$. In the case of genus~$2$ theta constants,
considering these three matrices~$M_1, M_2$ and~$N_{1,2}$ is enough to
run the Newton scheme, using the following
formulas~\cite[Prop.~8]{enge_ComputingClassPolynomials2014}:
\begin{equation}
  \label{eq:feedback-g2}
  \begin{aligned}
    \theta_{00,00}^2(0,M_1\tau) &= -\tau_{1,1}\theta_{01,00}^2(0,\tau), \\ 
    \theta_{00,00}^2(0,M_2\tau) &= -\tau_{2,2} \theta_{10,00}^2(0,\tau), \quad\text{and}\\ 
    \theta_{00,00}^2(0,N_{1,2}\tau) &= (\tau_{1,2}^2-\tau_{1,1}\tau_{2,2}) \theta_{00,00}^2(0,\tau).    
  \end{aligned}
\end{equation}

In~\cite{kieffer_SignChoicesAGMtoappear}, it is shown that all four
Borchardt sequences of the form~\eqref{eq:theta-as-borchardt-4} taken
at~$\tau$,~$M_1\tau$,~$M_2\tau$ and~$N_{1,2}\tau$ are given by good
choices of square roots only, provided that~$\tau$ satisfies the
following conditions:
\begin{itemize}
\item $\abs{\re(\tau_{i,j})} \leq \frac12$ for all~$1\leq i,j\leq 2$,
\item $2\abs{\im(\tau_{1,2})} \leq \im(\tau_{1,1}) \leq \im(\tau_{2,2})$,
\item $\abs{\tau_{j,j}}\geq 1$ for~$j = 1,2$.
\end{itemize}
These inequalities hold in particular whenever~$\tau$ lies in the
Siegel fundamental domain~$\Fund_2$.  Let~$\Red_2$ be the compact set
of such matrices~$\tau$, with the additional assumption
that~$\im(\tau_{1,1})\leq 2$ and~$\im(\tau_{2,2})\leq 8$. This choice
of upper bounds will be explained by the construction of a uniform
algorithm in~§\ref{sec:unif}.

As in the previous sections, we will collect the explicit data we need
to apply \cref{prop:good-borchardt-analytic} using inequalities
satisfied by genus~$2$ theta constants. Many such inequalities already
appear in~\cite[§9]{klingen_IntroductoryLecturesSiegel1990},
\cite[§6.2.1]{dupont_MoyenneArithmeticogeometriqueSuites2006},
\cite[§7.2]{streng_ComputingIgusaClass2014},
and~\cite{kieffer_SignChoicesAGMtoappear}; we will use one more.

\begin{lem}
  \label{lem:theta-inequalities-g2}
  For each~$\tau\in \Red_2$, we have
  $0.44 < \abs{\theta_{0,0}(0,\tau/2)} < 2.66$.
\end{lem}

\begin{proof} Let
  \begin{displaymath}
    \xi_0(\tau/2) = 1 + 2\exp\paren[\big]{i\pi \im(\tau_{1,1})/2}
    + 2\exp\paren[\big]{i\pi \im(\tau_{2,2})/2}.
  \end{displaymath}
  Since~$\tau\in\Red_2$, the complex number~$\xi_0(\tau/2)$ has
  modulus at least~$1$ and at
  most~$2.1$. By~\cite[Lem.~4.4]{kieffer_SignChoicesAGMtoappear}, we
  have $\abs{\theta_{0,0}(0,\tau/2) - \xi_0(\tau/2)} < 0.56$.
\end{proof}

\begin{prop}
  \label{prop:data-g2}
  Let~$\tau\in \Red_2$. Then, in the notation
  of~\emph{§\ref{subsec:borchardt-good}}, the following bounds apply.
  \begin{enumerate}
  \item In the case of the Borchardt
    sequence~\eqref{eq:theta-as-borchardt-4}
    with~$\lambda = \theta_{0,0}^2(0,\tau)$, we can take $m_0 = 0.069$
    and~$M_0 = 13$.
  \item In the case of the Borchardt
    sequence~\eqref{eq:theta-as-borchardt-4} taken at~$M_j\tau$
    with~$\lambda = \theta_{0,0}^2(0,M_j\tau)$, for
    each~$j\in\{1,2\}$, we can take~$m_0 = 9.7\cdot 10^{-7}$
    and~$M_0= 13$.
  \item In the case of the Borchardt
    sequence~\eqref{eq:theta-as-borchardt-4} taken at~$N_{1,2}\tau$
    with~$\lambda = \theta_{0,0}^2(0,N_{1,2}\tau)$, we can take
    $m_0 = 2.2\cdot 10^{-9}$ and~$M_0 = 13$.
  \end{enumerate}
\end{prop}

\begin{proof}
  We only have to analyze the first term of each of these Borchardt
  sequences. These explicit constants are then derived from the proof
  in~\cite{kieffer_SignChoicesAGMtoappear} that these complex numbers
  are in good position.
\end{proof}

\begin{thm}
  \label{thm:analytic-g2}
  Let $\rho = 1.9\cdot 10^{-23}$, define~$\Theta$ as
  in~\eqref{eq:theta-quotients} for~$g=2$, and let
  \begin{displaymath}
    \V = \bigcup_{\tau\in \Red_2} D_\rho\paren[\big]{\Theta(\tau)}
    \subset \C^3.
  \end{displaymath}
  Then the operations described in~\emph{§\ref{subsec:general-picture}},
  taking good choices of square roots always, define an analytic
  function~$F\from \V\to \C^2$ such that
  \begin{displaymath}
    F\paren[\big]{\Theta(\tau)}
       = (\tau_{1,1},\tau_{2,2}, \tau_{1,2}^2- \tau_{1,1}\tau_{2,2})
  \end{displaymath}
  for each~$\tau\in \Red_{2}$. We have~$\norm{F}\leq 4.5\cdot 10^4$
  uniformly on~$\V$.
\end{thm}

\begin{proof}
  The first terms of each of the Borchardt sequences analyzed in
  \cref{prop:data-g2} is obtained as quotients of the quantities
  \begin{displaymath}
    \frac{\theta_{a,b}^2(0,\tau)}{\theta_{0,0}^2(0,\tau/2)},    
  \end{displaymath}
  for all even theta characteristics~$(a,b)$ (i.e.~such that
  $a^tb = 0\mod 2$), except~$(11,11)$. The numerator and denominator
  of these quantities is bounded, both above and away from zero, by
  \cref{lem:theta-inequalities-g2}
  and~\cite[Cor.~7.7]{streng_ComputingIgusaClass2014}. Using these
  inequalities combined
  with~\cref{prop:data-g2,prop:good-borchardt-analytic} is sufficient
  to obtain an explicit value of~$\rho$.
\end{proof}

To conclude, we show that the Jacobian of~$F$ is uniformly invertible
by writing its inverse in terms of theta functions. Since~$F$ only
recovers the square of~$\tau_{1,2}$, we use the fact that each of the
fundamental theta constants~$\theta_{0,b}(0,\cdot)$ for~$b\in\I_2$ is
invariant under change of sign of~$\tau_{1,2}$.

\begin{lem}
  \label{lem:theta-from-square-g2}
  Let~$\V\subset\C^3$ be the image of~$\Half_2$
  under~$\tau\mapsto (\tau_{1,1},\tau_{2,2},-\det \tau)$. Then, for
  each~$b\in \I_2$, there exists a unique analytic
  function~$\xi_{0,b}\from \V\to \C$ such
  that
  \begin{displaymath}
    \theta_{0,b}(0,\tau/2) = \xi_{0,b}(
      \tau_{1,1},\tau_{2,2},-\det \tau)
  \end{displaymath}
  for all~$\tau\in \Half_2$.
\end{lem}

\begin{proof}
  In the theta series~\eqref{eq:theta-function} for~$z=0$, the only
  terms involving~$\tau_{1,2}$ are those associated
  with~$(n_1,n_2)\in\Z^2$ both nonzero. Write~$b = (b_1,b_2)$. Then,
  the terms associated with~$(n_1,n_2)$ and~$(n_1,-n_2)$ are
  \begin{displaymath}
    \exp\paren[\big]{i\pi(\tau_{1,1}n_1^2 + \tau_{2,2}n_2^2 \pm 2\tau_{1,2}n_1n_2)}
    (-1)^{n_1b_1 + n_2 b_2},
  \end{displaymath}
  so their sum can be written as a power series in~$\tau_{1,2}^2$
  only.
\end{proof}

Let~$G$ be the following analytic function:
\begin{displaymath}
  G(x,y,z) = \paren*{\frac{\xi_{0,b}(x,y,z)}{\xi_{0,0}(x,y,z)}}_{b\in \I_2\setminus\set{0}}.
\end{displaymath}
It is well-defined on a neighborhood of the image of~$\Red_2$
by~$\tau\mapsto (\tau_{1,1},\tau_{2,2}, -\det\tau)$, and is the
reciprocal of~$F$.

\begin{prop}
  We have~$\nind{dG
      (\tau_{1,1},\tau_{2,2},-\det\tau)}\leq 1.3\cdot 10^4$ for
  all~$\tau\in \Red_2$.
\end{prop}

\begin{proof}
  Fix~$\rho = 1/4$, and let us compute an upper bound
  on~$\abs{\xi_{0,b}(x,y,z)}$ for each point
  $(x,y,z) \in
  \D_\rho\paren[\big]{(\tau_{1,1},\tau_{2,2},-\det\tau)}$.
  Then~$x,y,z$ are of the form~$(\tau'_{1,1},\tau'_{2,2}, -\det\tau')$
  for some~$\tau'\in \Half_2$; the smallest eigenvalue of~$\im(\tau')$
  is bounded from below by
  \begin{displaymath}
    \frac{\det(\tau')}{\tr(\tau')} \geq 0.12.
  \end{displaymath}
  By the proof of~\cite[Lem.~4.7]{kieffer_SignChoicesAGMtoappear}, the
  function~$\abs{\xi_{0,b}}$ is uniformly bounded above by~$9.28$ on
  the disk we consider. By \cref{prop:cauchy}, we have
  \begin{displaymath}
    \nind{d\xi_{0,b}
        (\tau_{1,1},\tau_{2,2},-\det\tau)} \leq 149
  \end{displaymath}
  for each~$b\in \I_2$. The upper bound on~$\nind{dG}$ then
  follows from \cref{lem:theta-inequalities-g2}.
\end{proof}

\begin{cor}
  For all~$\tau\in \Red_2$, the Newton scheme described
  in~\emph{§\ref{subsec:general-picture}} to compute theta constants
  at~$\tau$ will converge starting from approximations
  of~$\Theta(\tau)$ to~$300$ bits of precision.
\end{cor}

\subsection{Higher genera}

In higher genera, including the case of genus~$2$ theta functions, we
are no longer able to show that the linear systems appearing in the
Newton schemes are invertible, nor a fortiori are we able to give an
explicit upper bound on the norm of their inverse Jacobians. Let us
shortly explain what the obstacle is.

In order to build a Newton scheme, the linearized system must be
square; however, as~$g$ grows, the number~$r$ of theta quotients
(either~$2^g-1$ in the case of theta constants, or~$2^{g+1}-2$ in the
case of theta functions) becomes greater than the dimension
of~$\Half_g$ or~$\C^g\times\Half_g$ respectively. Two ways around
this issue are suggested
in~\cite[§3.5]{labrande_ComputingThetaFunctions2016}:
\begin{enumerate}
\item One could keep all theta quotients as variables, and simply
  consider more symplectic matrices~$N$ to provide suitable feedback; or
\item One could perform Newton iterations not on the whole of~$\C^r$,
  but rather on the algebraic subvariety of~$\C^r$ obtained as the
  image of~$\Half_g$ or~$\C^g\times\Half_g$ by the fundamental theta
  quotients~\eqref{eq:theta-quotients}
  or~\eqref{eq:theta-quotients-2}.
\end{enumerate}

A fundamental obstacle to the second idea seems to be that the
algebraic subvariety of~$\C^r$ on which the theta quotients lie is not
smooth everywhere in general: consider for instance the Kummer
equation~\cite[§3.1]{gaudry_FastGenusArithmetic2007} in the case of
genus~$2$ theta constants. On the other hand, it seems very likely
that the first possibility can give rise to suitably invertible
systems, since much freedom is allowed in the choice of symplectic
matrices. However, the inverse of~$F$ will no longer be described
completely by theta functions, so the method we employed above to
prove the invertibility of the linearized systems no longer applies.

Despite the current lack of a uniform algorithm, the following
approach is available to certify the result of Newton's method to
evaluate theta constants (resp.~functions) at a given
$\tau\in \Half_g$ (resp.~$(z,\tau)\in \C^g\times\Half_g$). A finite
amount of precomputation, along with the results
of~\cref{sec:borchardt}, will allow us to compute real
numbers~$\rho>0$ and~$M>0$ such that the function~$F$ appearing in
Newton's method is analytic with~$\abs{F}\leq M$ on a polydisk of
radius~$\rho$ around the desired theta values. This gives upper bounds
on the norms of derivatives of~$F$ on a slightly smaller polydisk; in
particular, we can compute a certified approximation of~$dF^{-1}$
using finite differences, and check that it is indeed invertible. This
provides all the necessary data to run certified Newton
iterations.

\section{A uniform algorithm for genus~2 theta constants}
\label{sec:unif}

We have shown in~§\ref{subsec:g2} that genus~$2$ theta constants can
be evaluated on the compact subset~$\Red_2$ of~$\Half_2$ in uniform
quasi-linear time in the required precision, in a certified way. Using
this algorithm as a black box, we now design an algorithm to evaluate
genus~$2$ theta constants on the whole Siegel fundamental
domain~$\Fund_2$ in uniform quasi-linear time, generalizing the
strategy presented in~\cite[Thm.~5]{dupont_FastEvaluationModular2011},
\cite[§4.2]{labrande_ComputingJacobiTheta2018} in the genus~$1$ case:
we use duplication formulas to replace the input by another period
matrix which either lies in~$\Red_2$, or is sufficiently close to the
cusp, in which case the naive algorithm can be applied. We will use
the following transformations: for every $\tau \in\Half_2$, write
\begin{displaymath}
  D_1(\tau) = \frac\tau2
  \quad\text{and}\quad
  D_2(\tau) = \mat{2 \tau_{1,1}}{\tau_{1,2}}{\tau_{1,2}}{\frac12 \tau_{2,2}}.
\end{displaymath}
Recall that every~$\tau\in\Fund_2$ satisfies the following inequalities:
\begin{equation}
  \label{eq:F2}
  \begin{cases}
    \abs{\re(\tau_{i,j})} \leq \frac12 \quad \text{for each } 1\leq i, j\leq 2,\\[2pt]
    2\abs{\im(\tau_{1,2})} \leq \im(\tau_{1,1})\leq \im(\tau_{2,2}),\\[2pt]
    \abs{\tau_{i,i}}\geq 1 \quad \text{for each } i\in\{1,2\}. \phantom{\frac12}
  \end{cases}
\end{equation}
We also define
\begin{displaymath}
  \mathcal{J} = \paren[\big]{(00,00),(00,01),(10,00),(10,01)} \in (\I_2\times\I_2)^4,
\end{displaymath}
which is the tuple of theta characteristics corresponding to the
indices~$0,2,4,6$ in Dupont's
indexation~\cite[§6.2]{dupont_MoyenneArithmeticogeometriqueSuites2006}. For
each~$\tau\in \Half_2$, the duplication formula allows us to compute
all squares of theta constants at~$\tau$ given the theta
values~$\theta_{0,b}\paren[\big]{D_1(\tau)}$ for all~$b\in \I_2$. By
applying the theta transformation formula to the symplectic matrix
\begin{displaymath}
  \left(
    \begin{matrix}
      0&0&1&0\\0&1&0&0\\-1&0&0&0\\0&0&0&1
    \end{matrix}
  \right),
\end{displaymath}
we also see that all squares of theta constants at~$\tau$ can be
computed from the theta values~$\theta_{a,b}\paren[\big]{D_2(\tau)}$
for~$(a,b)\in\mathcal{J}$. It turns out that these complex numbers are
in good position; hence, they are easily determined from their squares
up to a harmless global change of sign.

\begin{lem}
  \label{lem:dupl}
  Let $\tau\in\Half_2$ be a matrix satisfying~\eqref{eq:F2}.
  \begin{enumerate}
  \item If~$D_1(\tau)$ satisfies~\eqref{eq:F2}, then the complex
    numbers~$\paren[\big]{\theta_{0,b}(D_1(\tau))}_{b\in \I_2}$ are in good position.
  \item If $D_2(\tau)$ satisfies~\eqref{eq:F2}, except that the real
    part of $D_2(\tau)_{1,1}$ is allowed to be smaller than~$1$
    instead of~$\frac12$, then the complex numbers
    $\paren[\big]{\theta_{a,b}(D_2(\tau))}_{(a,b)\in\mathcal{J}}$ are in good position.
  \end{enumerate}
\end{lem}

\begin{proof}
  See~\cite[Prop.~7.7]{streng_ComplexMultiplicationAbelian2010}
  and~\cite[Lem.~5.2]{kieffer_SignChoicesAGMtoappear}.
\end{proof}

\begin{thm}
  \label{thm:theta-F2}
  There exists an algorithm which, given $\tau\in\Half_2$
  satisfying~\eqref{eq:F2} and given $N\geq 1$, computes the squares
  of theta constants at~$\tau$ to precision~$N$ within
  $O(\M(N)\log N)$ binary operations, uniformly in~$\tau$.
\end{thm}

\begin{proof}
  Fix an arbitrary absolute constant~$C>0$ (for instance~$10$); in
  practice, this constant should be adjusted to minimize the
  algorithm's running time. First, let~$k_2$ be the smallest integer
  such that
  \begin{displaymath}
    2^{k_2} \im(\tau_{1,1}) \geq \min\{CN,\, 2^{-k_2-2} \im(\tau_{2,2})\},
  \end{displaymath}
  and let $\tau'$ be the matrix obtained after applying~$k_2$
  times~$D_2$ to~$\tau$ and reducing the real part at each step.  In
  order to compute theta constants at~$\tau$ to precision~$N$, we can
  compute theta constants at~$\tau'$ to some precision $N'\geq N$,
  then apply~$k_2$ times the duplication formula; all sign choices are
  good by \cref{lem:dupl}. We have $k_2 = O(\log N)$, and the total
  precision loss taken in extracting square roots is~$O(N)$
  by~\cite[Prop.~7.7]{streng_ComputingIgusaClass2014}. Therefore, the
  total precision loss is~$O(N)$ bits, and we can choose~$N' = C' N$
  where $C'$ is an absolute constant.

  Two cases arise now. If $\im(\tau'_{1,1}) \geq CN$, then we also
  have $\im(\tau'_{2,2})\geq CN$; therefore we can compute theta
  constants at~$\tau'$ to precision~$N'$ using~$O\paren[\big]{\M(N)}$
  operations with the naive algorithm. Otherwise, we have
  \begin{displaymath}
    \im(\tau'_{1,1}) \leq \im(\tau'_{2,2}) \leq 4 \im(\tau'_{1,1}) \leq 4CN.
  \end{displaymath}
  Therefore we can find an integer $k_1 = O(\log N)$ such that
  $\tau'' = D_1^{k_1}(\tau')$ belongs to~$\Red_2$, by definition of
  this compact set. We will compute theta constants at~$\tau''$ to
  some precision~$N''\geq N'$ using the Newton scheme described
  in~§\ref{sec:dupont}, then use the duplication formula~$k_1$
  times. Since~$O(1)$ bits of precision are lost each time we apply
  the duplication formula
  by~\cite[Prop.~7.7]{streng_ComputingIgusaClass2014}, we can also
  take~$N'' = C''N$ where~$C''$ is an absolute constant. Therefore,
  the whole algorithm can be executed in~$O(\M(N)\log N)$ binary
  operations.
\end{proof}

\begin{rem}
  In order to implement this algorithm in a certified way, one could
  use~\cite[Prop.~7.7]{streng_ComputingIgusaClass2014} more explicitly
  to track down an acceptable value of~$C''$. Another possibility is to
  start with~$C'' = 1.1$, say, and attempt to run this algorithm using
  interval arithmetic to obtain real-time upper bounds on the
  precision losses incurred. If the final precision we obtain is not
  satisfactory, we may simply double~$C''$ and restart. The resulting
  algorithm still has a uniform quasi-linear cost.
\end{rem}

\bibliographystyle{abbrv}
\bibliography{theta-unif}

\end{document}